\documentclass[11pt]{amsart}
\usepackage{amsmath}
\usepackage{amssymb}
\usepackage{amscd, mathrsfs}
\usepackage{amsthm}
\usepackage{hyperref}

\usepackage{amsfonts}
\usepackage{subfigure}

\usepackage[all]{xy}
\usepackage{color}

\numberwithin{equation}{section}

\newtheorem{theorem}[equation]{Theorem}

\newtheorem{proposition}[equation]{Proposition}

\newtheorem{lemma}[equation]{Lemma}

\newtheorem{corollary}[equation]{Corollary}

\newtheorem{conjecture}[equation]{Conjecture}

\newtheorem{rem}[equation]{Remark}

\theoremstyle{definition}
\newtheorem{definition}[equation]{Definition}

\newtheorem{example}[equation]{Example}
\newtheorem{remark}[equation]{Remark}

\def\XXint#1#2#3{{\setbox0=\hbox{$#1{#2#3}{\int}$} 
	\vcenter{\hbox{$#2#3$}}\kern-.5\wd0}}

\newcommand{\N}{\mathbb N}

\newcommand{\Cl}{\mathrm {Cl}}

\newcommand{\R}{\mathbb R}

\newcommand{\Q}{\mathbb Q}
\newcommand{\G}{{\mathbb G}}

\newcommand{\Int}{\operatorname{Int}}

\newcommand{\Span}{\operatorname{span}}
\newcommand{\dd}{\operatorname{d}\hspace{-0.07cm}}

\def\eps{\epsilon}

\newcommand{\Ad}{\operatorname{Ad}}
\newcommand{\ad}{\operatorname{ad}}

\newcommand{\vol}{\operatorname{vol}}
\newcommand{\uno}{{\mathbb 1}}
\usepackage{bbold}

\begin{document}

\title[Sets with constant  normal in Carnot groups]{
Sets with constant  normal in Carnot groups: properties and 
examples}

%
%
\author[C. Bellettini and E. Le Donne]{Costante Bellettini and Enrico Le Donne}
\address{\textsc{Costante Bellettini}: \\
Department of Mathematics, University College London, London WC1E 6BT, United Kingdom}
\email{c.bellettini@ucl.ac.uk}
\address{\textsc{Enrico Le Donne}: \\
Dipartimento di Matematica, Universit\`a di Pisa, Largo B. Pontecorvo 5, 56127 Pisa, Italy \\
\& \\
University of Jyv\"askyl\"a, Department of Mathematics and Statistics, P.O. Box (MaD), FI-40014, Finland}
\email{enrico.ledonne@unipi.it}


\renewcommand{\subjclassname}{%
\textup{2010} Mathematics Subject Classification}

\date{\today}

\renewcommand{\subjclassname}{%
 \textup{2010} Mathematics Subject Classification}

 \thanks{
 C. B. was partially supported by the EPSRC grant EP/S005641/1 and by the National Science Foundation Grant No. DMS-1638352.
 E.L.D. was partially supported by the Academy of Finland (grant
288501
`\emph{Geometry of subRiemannian groups}' and by grant
322898
`\emph{Sub-Riemannian Geometry via Metric-geometry and Lie-group Theory}')
and by the European Research Council
 (ERC Starting Grant 713998 GeoMeG `\emph{Geometry of Metric Groups}').
}
\begin{abstract}
We analyze subsets of Carnot groups that have intrinsic constant normal, as they appear in the blowup study of sets that have finite sub-Riemannian perimeter. 
The purpose of this paper is threefold.
First, we prove
 some mild regularity and structural results		   in arbitrary Carnot groups.
 Namely, we show that for every constant-normal set in a Carnot group 
 its sub-Riemannian-Lebesgue representative is  
 regularly open,   contractible, and 
 its topological boundary coincides with the reduced boundary and with the measure-theoretic boundary.
We infer these properties from a cone property. Such a cone will be a semisubgroup with nonempty interior  that is canonically associated with the normal direction. We characterize the constant-normal sets exactly as those that are arbitrary unions of translations of such semisubgroups. 
  Second, making use of such a characterization, we provide some pathological examples in the specific case of the free-Carnot group of step 3 and rank 2.
  Namely, we  construct a
     constant normal set that, with respect to any   Riemannian metric, is not of locally finite perimeter;  we also construct an example with non-unique intrinsic blowup at some point, showing that it has
     different upper and lower sub-Riemannian density at the origin.
Third, we show that in Carnot groups of step 4 or less, every constant-normal set is intrinsically rectifiable, in the sense of Franchi, Serapioni, and Serra Cassano. 
  
\end{abstract}

 \maketitle

\setcounter{tocdepth}{2}
\tableofcontents


\section{Introduction}
   Subsets of Carnot group whose intrinsic  normal is constantly equal to a left-invariant vector field appear both in the development of a theory \`a la De Giorgi for sets with locally finite-perimeter in sub-Riemannian spaces \cite{fssc2, fssc, 
   MR3353698} 
    and in the obstruction results for biLipschitz embeddings into $L^1$ of non-Abelian nilpotent groups \cite{Cheeger-Kleiner2}.
  The work \cite{fssc} by Franchi, Serapioni and Serra-Cassano provides complete understanding of such sets in the case of Carnot groups with nilpotency step 2. However, in higher step the study appears to be much more challenging, due to the more complex underlying algebraic structure, and only  in the case of filiform groups we have a satisfactory understanding of sets with constant intrinsic normal, see \cite{Bellettini-LeDonne}.

    The main result of this paper, valid in \textit{arbitrary} Carnot groups, is summarized by the following claim, which provides some mild regularity information and some basic topological properties.
\begin{theorem}
\label{main:thm}
 Every constant-normal set in a Carnot group admits a representative that is regularly open, is contractible, and admits a cone property.
\end{theorem}   
   
 By means of explicit examples, we will see that further regularity properties (which are for example valid in filiform groups) may fail to hold in some Carnot groups. For example, a constant-normal set is not necessarily a set of locally finite perimeter in the Euclidean sense: we will illustrate this in the case of the free Carnot group of step 3 and rank 2, see Section~\ref{sec:controesempio1}.

A more precise (and stronger) version of Theorem~\ref{main:thm} is given in Theorem~\ref{main:thm2}. As immediate consequences of it, we will obtain that every constant-normal set in a Carnot group admits a representative whose topological boundary coincides with the measure-theoretic boundary, with the De Giorgi's reduced boundary, and with the support of its perimeter measure, see Proposition~\ref{remark:normalratioconstant}. Moreover, we will deduce that the density ratio (independently of the radius) at boundary points is pinched between constants that only depend on the group, see Proposition~\ref{prop:densities}. 

Working again in the free Carnot group of step 3 and rank 2, we will provide an example of constant-normal set for which there exist boundary points at which the (intrinsic) lower density is \textit{strictly smaller} than the upper density: this implies in particular that the (subRiemannian) tangent cone at such points is \textit{not unique}, see Section~\ref{sec:controesempio2}. It is natural to ask how large the set of such boundary points can be. We answer this question in Section~\ref{rectifiabilityStep4} as a consequence of the following result: in $\mathbb F_{23}$ (or, more generally, in any other Carnot group of step at most 4) constant normal sets are \textit{intrinsically rectifiable} in the sense of Franchi-Serapioni-Serra Cassano. Indeed, we show that
we only have half-spaces as blowups at almost every point with respect to the sub-Riemannian perimeter measure (so the density is $1/2$ at almost every point).

 
We now explain the terminology and present the result in more detail.
  Let $\G $ be a Carnot group;
  we refer to \cite{LeDonne:Carnot} for an introduction to Carnot groups and their sub-Riemannian geometry.
   We denote by $V_1$ the first layer of its stratification and by $\delta_t$ the group automorphism induced by the multiplication by $t$ on $V_1$, which is a dilation by $t$ with respect to the Carnot distance. Elements of $V_1$ are seen as left-invariant vector fields. 
  
    Let $E\subseteq \G$ be measurable set and $X $ a left-invariant vector field on $\G$. We say that $E$ is {\em $X$-monotone} if 
\begin{equation} X\uno_{E}\geq 0, \end{equation} 
in the sense of distributions, see \eqref{def:Xuno} and Remark~\ref{remark:positivedistr}. 
  We say that $E$ is  {\em precisely $X$-monotone} if  
  \begin{equation}   E \exp(tX)\subseteq     E,\qquad \forall t>0.
  \end{equation} 
We say $E$ is a   {\em constant-normal set} (resp.,  {\em precisely constant-normal set}) if there exists a closed half-space $W$ in $V_1$ such that for all $X\in W$ the set $E$ is $X$-monotone (resp., precisely $X$-monotone). 
  After fixing a scalar product $\langle \cdot,\cdot\rangle$ on $V_1$, if $X\in V_1$ is such that $W= \{ \langle \cdot,X\rangle \geq0 \}$ and $ \langle X,X\rangle=1$, we say that $X$ is the  {\em constant normal}, resp. {\em precisely constant normal} of $E$.


We say that a subset $C\subset \G$ of a Carnot group $\G$   is  a {\em cone} if it is non-empty and
\begin{equation}
\label{cone_def}\delta_t(C) = C, \qquad \forall t>0.
\end{equation}
If $E\subset \G$  is measurable  and $C\subset \G$ is a cone, we say that $E$ has the $C$-{\em cone property} if
\begin{equation}
\label{cone_property}
   E \cdot C \subseteq     E.
\end{equation}

While it is obvious that every precisely constant-normal set has constant normal (see Remark~\ref{rmk:constant:constant_precise}), it is not immediate that every constant-normal set admits a   representative that  has precisely constant normal (see Remark~\ref{rmk:constant_precise}).

In order to explain Theorem~\ref{main:thm} we describe the construction of the representative and of the cone. We shall see the theorem as a consequence of  
the fact that precisely constant sets normal have a cone property and then prove that every set with this cone property has  a representative as in Theorem~\ref{main:thm}.

The Carnot group $\G$ is equipped with a Carnot distance, i.e., a sub-Riemannian left-invariant distance where the tangent subbundle is induced by the first layer of the stratification. See  \cite{LeDonne:Carnot} for the notion of  stratification.
 Equipped with such a Carnot distance and with any Haar measure, the Carnot group $\G$ becomes   a doubling metric measure space. Hence the Lebesgue-Besicovitch Differentiation Theorem holds. Consequently, to every measurable set $E$ we associate the  {\em Carnot-Lebesgue  representative} of  
$E$, that is, the set of points of (intrinsic) density 1 for $E$.

If $W$ is a closed half-space in $V_1$, we consider
 the semigroup $S_W$ 
generated by $\exp(W)$.
It is crucial that $S_W$ is a cone that is contractible and whose interior is not empty. This last  fact 
is a consequence of the Orbit Theorem since the set $W$ Lie generates the Lie algebra.

 The cone in the next result is the interior $C_W$ of 
 the semigroup $S_W$ given by the $W$ from the definition of the constant normal.
We do not need to change representative. 
\begin{theorem}
\label{thm1}
Every precisely constant-normal set has the cone property  with respect to an open   contractible  cone.
\end{theorem}

Regarding the next result we recall that a set is regularly open if it equals   the interior of its closure. 
\begin{theorem} 
\label{thm2}
If a set has the  cone property  with respect to an open  
cone, then its  Carnot-Lebesgue  representative   is regularly open, contractible,  and with contractible complement.
 \end{theorem}

In order to prove Theorem~\ref{main:thm}, in Theorem~\ref{main:thm2} we shall prove that the Carnot-Lebesgue representative of a set with constant normal $W$ has the cone property with respect to the closure of $S_W$. 

Theorem~\ref{thm1} has  an important counterpart. In fact, every set that has the cone property with respect to  $S_W$ has  
precisely constant-normal.
This observation let us construct precisely constant-normal sets (all of them) as union of arbitrary  translations of the semigroup:
for every set $\Sigma\subset \G$ 
and for every family $(E_p)_{p\in \Sigma}$ of sets that have precisely constant-normal given by the half-space $W$, we have that 
the set 
$$E:= \bigcup_{p\in \Sigma} p E_p $$
has precisely constant-normal given by the half-space $W$; see Proposition~\ref{monotone_semigroup}.
For example, one can take as $E_p$ either of the sets  $ C_W$,  $S_W$ or their closures.

For the reason that the semigroup   $S_W$ play such an important role in the theory of constant normal sets, it becomes fundamental to understand its geometry.
An unexpected result is that  the semigroup   generated by an open (or closed)  half-space $W$ in the horizontal space is not   open nor  closed in the whole group $\G$, in general. And moreover, the subset $\exp(W)$ may all be contained in the topological boundary of the semigroup.

The simplest Carnot group where one sees these pathologies is the Carnot group whose Lie algebra is  free-nilpotent of rank 2 and step 3. We call such a Carnot group $\mathbb F_{23}$ and we denote by $X_1$ and $X_2$ the generating horizontal left-invariant vector fields, which we can make orthogonal by a suitable choice of scalar product. Being a free-nilpotent group, any two horizontal left-invariant vector fields differ by a Carnot automorphism. Hence, there is no loss of generality in studying only the sets that have constant normal $X_2$.
We can completely characterize the semigroup $S$ generated by 
$$\exp(W):=\{\exp(aX_1+b X_2): \quad a\in \R, b\geq 0\}.$$ 
We shall see that $\exp(W)$ is contained in the boundary of $S$. Also, inside the set $S$ there is no  open subset that is an Euclidean cone in exponential coordinates. From this fact we are able to construct examples of precisely constant normal sets that do not have  locally finite Riemannian perimeter, see Theorem~\ref{thm:no_BV}. 
We also construct an example with different (intrinsic) upper and lower density at some point; see Theorem~\ref{prop:non_unique_tangents}


\subsubsection*{Acknowledgements} The authors are grateful to Sebastiano Don and Terhi Moisala for their constructive feedback on previous versions of this paper.
Part of this work was carried out while the first-named author was member of the Institute for Advanced Study (Princeton) in 2013 and 2019.
Both authors gratefully acknowledge the excellent research environment and hospitality.

\section{Sets with constant   normal, representatives, and cone properties}

The general argument that we will pursue 
is to consider the semigroup $S=S_W$ generated by an open (or closed)  half-space $W$ in the horizontal layer. The subset $S$ will   be neither open nor  closed, in general. However, the set $S$ has nonempty interior and is a cone (see Proposition~\ref{semigroup:has:interior} and Corollary~\ref{semigroup_Canot}). Consequently, we have that $S$  and $C:=\Int(S)$ are almost regular cones (see Definition~\ref{def:regular} and Lemma~\ref{lem:almost regular}); in addition, we have that the closure $\bar S$ equals  the closure $\bar C$ and is regularly closed (see Definition~\ref{def:regular}).

We shall then consider  a set $E$ that has constant normal with respect to $W$. We first prove that $E$ has a Lebesgue representative that has the $S$-cone property (see Corollary~\ref{monotone_subgroup_Carnot}). Secondly, we prove that the Lebesgue representative of $E$ with respect to the Carnot distance is open and has the  $\bar S$-cone property (see Lemmas~\ref{lem:CL:open} and \ref{lem:closure_cone_prop}).

\subsection{Some topological notions}
Here is some topological terminology that we will use.
We shall say that a subset {\em has interior} if its interior is not empty. Fixed a 
topological space $G$, for example 
a Lie group,
 we shall denote by $\Cl(E)$ and  $\Int(E)$ the closure and the interior, respectively, of a subset $E\subseteq G$ within the topological space $G$.  We may also write $\bar E$ for $\Cl(E)$.

\begin{definition}\label{def:regular}
Let   $C$ be a subset of a topological space $G$. 
The subset $C$   is said to be {\em regularly open}  if
$ C = \Int ( \Cl (C)) $.  
It is said to be {\em regularly closed} if $C = \Cl (\Int (C)).$
The subset $C$   is {\em almost regular}  if 
\begin{equation}\label{def:almost regular} C \subseteq \Cl (\Int (C));\end{equation}
equivalently, if $ \Cl(C) = \Cl (\Int (C)).$ {The terminology `almost regular' is not standard. }
\end{definition}
\begin{remark}\label{remark:regularly}
\hspace{-0.1cm}{\bf a.} If $C$ is almost regular then its closure is regularly closed, while the opposite   is not true.
Indeed, if $C$ is almost regular then 
$$\bar C  = \Cl (\Int (C)) \subseteq  \Cl (\Int (\bar C)) \subseteq \Cl  (\bar C) = \bar C.$$
Vice versa, a counterexample is given by $\Q\cap [0,1]$ as subset of $\R$.

\ref{remark:regularly}.b. If $C$ is almost regular then also $\Int(C$) is almost regular.


\ref{remark:regularly}.c. Every almost regular cone has interior. Indeed, if the interior of such a set would be empty, then the almost-regular assumption would imply that the set is empty. However, by definition
  cones are assumed nonempty. 
\end{remark}
\subsection{Lebesgue representatives} Let $G$ be a Lie group, not necessarily a Carnot group,  endowed with a left- and right-invariant Haar measure $\vol$, which is fixed along the conversation.
  Given a distance $\rho$ on $G$ inducing the topology and a
  measurable set $E\subseteq G$, we denote by 
    $\tilde E^\rho$   the Lebesgue $\rho$-representative of $E$ with respect to    $\rho$ (and $\vol$). Namely,  
\begin{equation}
\label{def:tilde:E}
x \in \tilde E^\rho
  \quad
  \Leftrightarrow
  \quad 
  \lim_{r \to 0^+} \frac{\vol(B_\rho(x,r) \cap E)}{\vol(B_\rho(x,r))} =1,
\end{equation}
where $B_\rho(x,r):=\{y\in G \mid \rho(x,y)< r\}$.
We stress that,   even in   Euclidean spaces, the set $\tilde E$  is generally neither  open nor is closed.

When the distance $\rho$ is sub-Riemannian, then the metric measure space $(G, \rho, \vol)$ is locally doubling. Every time such a space is locally doubling, 
 the Lebesgue-Besicovitch Differentiation Theorem holds and in particular the sets $\tilde E^\rho$ and $E$ agree up to a $\vol$-negligible set.
Namely, we have that $\tilde E^\rho$ is a representative of $E$, where we shall say that $E'$ is a {\em representative} of $E$ if  
$\vol(E\setminus E')=\vol( E'\setminus E)=0$  
 
\subsection{Monotone directions}
 In this subsection we discuss equivalent definition of monotone directions and show that if $E$ is $X$-monotone, then there is a representative 
 of $E$ that is precisely  $X$-monotone.
 In the discussion, we fix a Lie group $G$ and a left-invariant vector field $X$ on $G$.

Recall that a distribution $T$ on $C^\infty_c(G)$ is said to be {\em non-negative}, and we write $T\geq 0$, if
for all test functions $u$ such that $u\geq0$ we have $Tu\geq 0 $. In particular, 
we say that, in the sense of distributions,  
$X\uno_{  E}\geq 0$ if for all $u\in C^\infty_c(G)$ with $u\geq 0$, one has that \begin{equation}
\label{def:Xuno}
\langle X \uno_{E}, u\rangle := -\int_{\tilde E} Xu \dd\vol \geq 0.\end{equation}

\begin{remark}
 \label{remark:positivedistr}
It is a general fact that non-negative distributions are Radon measures, see \cite[Theorem 1.39]{Evans_Gariepy_revised}. Here is the short argument with our notation. For $T\geq 0$ on $G$ and $K\subset G$ compact, fix a function $U_K \in C_c^\infty(G)$ such that $U_K=1$ on $K$ and $U_K\geq 0$. Let $u \in C_c^\infty(G)$ with $\text{spt}(u) \subset K$ and $C:=\|u\|_{L^\infty}$. Then $-C U_K \leq u \leq CU_K$, in other words $C U_K-u$ and $u+C U_K$ are non-negative functions in $C_c^\infty(G)$: by the assumptions $T(C U_K-u)\geq 0$ and $T(u+C U_K)\geq 0$ and the linearity of $T$ gives $C T(U_K)\geq T(u)$ and $T(u)\geq -C T(U_K)$. These inequalities yield a constant $A_K =T(U_K) \geq 0$ (depending only on $K$) such that $|T(u)|\leq A_K \|u\|_{L^\infty}$. The distribution $T$ extends therefore to a bounded linear functional on $C_c^0(G)$ and it is represented by a non-negative Radon measure by Riesz' theorem (the non-negativity is preserved in the extension).
\end{remark}

With next result we give equivalent characterizations of the property that a set $E\subseteq G$ is {\em $X$-monotone}. It is important to recall that in every Lie group $G$   the flow $\Phi^t_X$ at time $t$ with respect to a  left-invariant vector field $X$ is the right translation:
\begin{equation}
\Phi^t_X(p) = p\exp(tX),\qquad \forall t\in \R, \forall p\in G.
\end{equation}
Initially, to better work with such right translations, we shall consider Lebesgue representative for some {\it right}-invariant distance.
In the next result, as an example of  right-invariant doubling distance  $\rho$ one may consider any right-invariant Riemannian distance.

\begin{proposition}\label{monotone_direction} Let $E\subseteq G$ be a measurable set of a Lie group $G$, $X$  a left-invariant vector field on $G$, and $\rho$ a right-invariant doubling distance for $\vol$.
 The following are equivalent:
\begin{itemize}
\item[(\ref{monotone_direction}.0)] in the sense of distributions, $X\uno_{E}$ is a Radon measure and  $X\uno_{E}\geq 0;$
\item[(\ref{monotone_direction}.1)] in the sense of distributions, $$X\uno_{E}\geq 0;$$
\item[(\ref{monotone_direction}.2)] for all $t>0$  we have that  $$  \uno_{E\exp(tX)}\leq  \uno_E,\qquad \text{almost everywhere};$$ 
\item[(\ref{monotone_direction}.3)]  the representative $  \tilde E ^\rho$, as defined in \eqref{def:tilde:E}, satisfies 
 the set inclusion  $$  \tilde E^\rho \exp(tX)\subseteq   \tilde E^\rho,\qquad\forall t>0;$$ 
\item[(\ref{monotone_direction}.4)] there exists a representative $E'$ of $E$  that satisfies 
 the set inclusion  
 $$    E' \exp(tX)\subseteq     E',\qquad\forall t>0.$$ 
\end{itemize}
\end{proposition}
\proof
Obviously (\ref{monotone_direction}.0) $\implies$ (\ref{monotone_direction}.1).
The fact that (\ref{monotone_direction}.1) $\implies$ 
(\ref{monotone_direction}.0) is shown in
 Remark~\ref{remark:positivedistr}.

The implications (\ref{monotone_direction}.1) $\implies$  (\ref{monotone_direction}.2) $\implies$  (\ref{monotone_direction}.3)  are in the previous work \cite[Lemma~2.7 and Lemma~2.8]{Bellettini-LeDonne}.
The implication (\ref{monotone_direction}.3) $\implies$ (\ref{monotone_direction}.4)
is obvious since $\tilde E^\rho$ is a representative of $E$.

We are left to prove that   (\ref{monotone_direction}.4) $\implies$  (\ref{monotone_direction}.1).
Since 
$E=E'$ a.e., then as distributions
$X\uno_{ E}=  X\uno_{ E'}$.
To show that $X\uno_{  E'}\geq 0,$
we perform the following calculation, for every smooth function $u$ with compact support and $u\geq 0$:
\begin{eqnarray*}
\langle X\uno_{E'}, u\rangle&=&-\int_{E'} Xu \dd\vol\\
&=&  -\int_{E'} \lim_{t\to 0} \dfrac{u(p\exp(tX))- u(p)}{t} \dd\vol(p)   \\
&=&  - \lim_{t\to 0}  \dfrac{1}{t} \left( \int_{E'} u(p\exp(tX))\dd\vol(p) -\int_{E'} u(p) \dd\vol(p)   \right)\\
&=&  - \lim_{t\to 0}  \dfrac{1}{t} \left( \int_{{E'}\exp(tX)}  u(p) \dd\vol(p )-\int_{E'} u(p) \dd\vol(p)   \right)\\
&=&  - \lim_{t\to 0}  \dfrac{1}{t} \left( -\int_{{E'}\setminus {E'} \exp(tX)}  u(p)\dd\vol(p)    \right)\quad \geq 0,
\end{eqnarray*}
where the limit and the integral could be swapped by the Dominated Convergence Theorem, since $u(p\exp(tX))$ is smooth both in $p$ and in $t$, and has compact support.  
\qed

\begin{remark}\label{rmk:constant:constant_precise}
 
To justify the previous proposition, we stress that 
the implication
 (\ref{monotone_direction}.4) $\implies$  (\ref{monotone_direction}.1) rephrases with the statement that
 every precisely constant-normal set has constant normal. Regarding the opposite implication we refer to Remark~\ref{rmk:constant_precise} below. 
 
\end{remark}

\subsection{Cone property}

We next assume that $\G$ is a Carnot group. In the introduction we defined the notion of cones in $\G$ in \eqref{cone_def}. We stress that there are no a priori assumptions on openness or closure of cones, nor do we assume that the identity is or is not in the cone.

We rephrase condition \eqref{cone_property} from the introduction saying that a set
$E\subset \G$    has the $C$-{\em cone property}, or better the \emph{inner $C$-cone property}, if
\begin{equation}
\label{cone_property2}
 p\in E\implies   p \cdot C \subseteq     E.
\end{equation}
In the setting of Carnot groups, recently two   cone properties have been introduced. In fact, in \cite{DLMV} the following terminology is used.
A set $\Gamma\subseteq \G$ satisfies the \emph{outer $C$-cone property} if 
\begin{equation}
	\label{cone_property2outer}
 	\Gamma\cap \Gamma C=\emptyset.
\end{equation}
This property is in similarity with the inner cone property \eqref{cone_property2}, which says 	$E\cap E C=EC.$
The outer cone property is satisfied by the topological boundary $\Gamma=\partial E$ whenever  $E$ satisfies the the inner cone property with respect to an open cone.  
In this last paper we should consider only the (inner) cone property \eqref{cone_property2}.

\begin{remark}
\label{cone_prop_complement}
  If $E$ has the $C$-cone property, then $E^c:=\G\setminus E$ has the $C^{-1}$-cone property.
Indeed, by contradiction, assume that there exists $p\in  E^c$ such that it is not true that $p C^{-1}$ is in $ E^c$. Namely,
there exists $q\in p C^{-1} \cap E$.  Therefore, $q^{-1} p \in C$ and so
$$p= q q^{-1} p\in q C\subseteq E,$$
where at the end we used that $E$ has the $C$-cone property. We reached a contradiction with $p\in  E^c$.
\end{remark}
 \begin{remark}
\label{cone_prop_open}
The interior of a cone is a cone, unless it is empty. In particular,
every cone with interior contains an open cone. Consequently, if a
 a set has the  cone property  with respect to a       cone with interior, then it also 
  has the  cone property  with respect to an open      cone.  
 \end{remark}
\begin{proposition}
\label{prop:contractible}
If a set $E\subseteq \G$ has the $C$-cone property with $Int(C)
\neq \emptyset$, then $E$ is contractible.
\end{proposition}
\begin{proof}
Up to left translating, we assume $1_\G\in E$.
Fix an auxiliary right-invariant Carnot distance $\rho$.
Since $C$ has interior, up to dilating we assume that for some $q_0\in E$ we have that the closed unit ball at $q_0$ satisfies
\begin{equation}\label{ball:in:C}
\bar B_\rho (q_0, 1)
\subseteq C.
\end{equation}
We shall consider the following retraction:
$\phi:\G\times[0,2] \to \G$
defined as for $(g,t)\in \G\times[0,1]$
$$\phi(g,t):= g \delta_t (g^{-1} \delta_{|g|} (q_0)),$$
where $|g|:=\rho(1_\G, g)$, and for $(g,t)\in \G\times[1,2]$
$$\phi(g,t):=  \delta_{2-t} ( \delta_{|g|} (q_0)).$$
We observe that $\phi$ is continuous, since for $t=1$
$$g \delta_1 (g^{-1} \delta_{|g|} (q_0))  = \delta_{|g|} (q_0)  = \delta_{2-1} ( \delta_{|g|} (q_0)),$$
because $\delta_1$ is the identity map.

The map $\phi$ is a retraction. Indeed, since $\delta_0$ is the  map  that is constantly equal to $1_\G$, we have that at $t=0$ the map $\phi$ is the identity on $\G$ and at $t=2$ it is constantly equal to $1_\G$.

We check now that if $g
\in E$, then $\phi(g,t)\in E$ for all $t\in  [0,2]$.
Assume first that $t\in  [0,1]$. We claim that
\begin{equation}\label{13_5_19A}
\phi(g,t) \in g C \subseteq E,\qquad \text{ for } t\in  [0,1],
\end{equation}
where the last containment holds since $E$ has the $C$-cone property. Regarding the inclusion \eqref{13_5_19A}, proving that
$g \delta_t (g^{-1} \delta_{|g|} (q_0))  \in g C $ is equivalent to
\begin{equation}\label{13_5_19B}
 \delta_{|g|^{-1}} (g^{-1} \delta_{|g|} (q_0))  \in  C ,\end{equation}
where we have used that $C$ is a cone.
We shall show this last inclusion using \eqref{ball:in:C}.
Indeed, using right-invariance we have
\begin{eqnarray*}
\rho(q_0,
 \delta_{|g|^{-1}} (g^{-1} \delta_{|g|} (q_0))  ) &=&
 \rho(q_0,
 \delta_{|g|^{-1}} (g^{-1} )  q_0  )\\
  &=&
 \rho(1_\G,
 \delta_{|g|^{-1}} (g^{-1} )    )\\
   &=&
 |g|^{-1}\rho(1_\G,
 g^{-1}     ) = |g|^{-1}\rho(
 g ,     1_\G ) =1.\\
\end{eqnarray*}
Hence \eqref{13_5_19A} and \eqref{13_5_19B} are proved.

Assume then $t\in [1,2]$. Since $C$ is a cone and $q_0\in E$,  we immediately have that
$$\phi(g,t) = \delta_{2-t} ( \delta_{|g|} (q_0)) \in C\subseteq E,$$
where at the end we used that $1_\G\in E$.

Hence, the map $\phi$ is a (continuous) retraction of $E$ onto $1_\G$.
\end{proof}

Most probably, one may improve the above result and show that the interior of such a set is homeomorphic to  an open ball. We do not go in this direction. However, we suggest to use a criterion by Stallings \cite{Stallings1962}:
a contractible open subset of  $\R^n$ that is simply connected at infinity is homeomorphic to $\R^n$. 
Recall that  $X$   `simply connected at infinity' means that for each compact $K$ of $X$ there is a larger compact $L$ such that the induced map on $\pi_1$ from $X \setminus L$ to $X \setminus K$ is trivial.

 We will not make use of the next remark, which can be used to directly prove that sets with precisely constant normal sets are connected. The reason is that in Theorem~\ref{main:thm2} we shall prove more: they are contractible.
\begin{remark} 
If $C$ is an open cone and $K\subseteq \G$ is a bounded set, then
  $\cap_{g\in K} gC  \neq \emptyset$.  
Here is the simple proof.
Fix $p\in  C$.
Since $C$ is open, the set $p^{-1}C$ is a neighborhood of the identity element of $\G$.
Since   $K $ is a bounded set, for every fixed Carnot distance $\rho$  there exists $R>0$ such that for all $g\in K$ we have
$\rho(1, g)<R<0$.
 Choose $\eps>0$ such that 
 $\delta_{\eps}B_\rho(1,R)  \subseteq  C^{-1}p$.
 Consequently, for every $g\in K$ we have 
 $\delta_{\eps}g \in  C^{-1}p$, which implies $p \in (\delta_{\eps} g) C$. 
We deduce  that 
 the point $\delta_{\frac{1}{\eps}} p $ is in $\delta_{\frac{1}{\eps}}((\delta_{\eps} g) C) = g \delta_{\frac{1}{\eps}} C = g C$, the latter equality uses that $C$ is a cone. 
  Therefore, we infer that $\delta_{\frac{1}{\eps}} p \in   gC$ for every  $g\in K$.
\end{remark}

\subsection{Semigroups}
In the paper, we use the standard terminology and say that a subset of a group is a {\em semigroup} (a better term is {\em semisubgroup}) if it is closed under multiplication. It is evident that every subset is contained in a semigroup 
with the property of being the smallest semigroup containing the set.

We next provide some basic results about semigroups in Carnot groups.
\begin{remark}\label{rmk:int:semigroup}
Both the closure and the interior of a semigroup in a topological group are  semigroups. Indeed, if $S$ is a semigroup and $p_n$, $q_n\in S$ are sequences such that $p_n\to p$, $q_n\to q$ then $p_nq_n $ is a    sequence in $S$ such that $p_nq_n\to pq$. So $\Cl(S)$ is a semigroup.
If instead we take $p,q\in \Int(S)$ then there exists $U$ open set with $q\in U\subseteq S$. Therefore, the set $pU$ is a neighborhood of $q$ and is contained in $S$. So $\Int(S)$ is a semigroup
\end{remark}

\begin{remark}\label{rmk:semi_coneprop} If $S$ is a semigroup, then we trivially have
\begin{equation}
\label{cone_property3}
 p\in S\implies   p \cdot S \subseteq     S.
\end{equation}
This implies that, if in addition $S$ is a cone, then $S$ has the $S$-cone property.
\end{remark}
Along this section, we shall say that a set is a {\em conical semigroup} if it is a semigroup and a cone.
From Proposition~\ref{prop:contractible} we know that if $S$ is a conical semigroup  with interior, then it is contractible. Next lemma says that for semigroups the statement is true regardless of having interior.
\begin{lemma}\label{lem:contractible}
In Carnot groups every  conical semigroup   is contractible. 
\end{lemma}
\begin{proof}
Let $p,q \in S$. The curve $t \mapsto\delta_t q$ for $t\in (0,1]$ extends by continuity at $t=0$ with value the identity of the group. Moreover, since $S$ is a cone, every point $\delta_t q$ for $t>0$ lies in $S$. 
Consider the curve $t\mapsto p (\delta_t q)$, which maps into $S$ by the semigroup property, since $p\in S$ and $\delta_t q \in S$. This curve, extended by continuity at $t=0$, is continuous and maps to $p$ for $t=0$ and to $pq$ for $t=1$.
Reversing the roles of $p$ and $q$ we obtain a curve $t\mapsto (\delta_t p) q$ (extended by continuity at $t=0$) that lies in   $S$ and connects $q$ to $pq$. Composing the first curve with the reverse of the second, we obtain a curve from $p$ to $q$. 
Such a curve from $p$ to $q$ depends continuously on $q$. Hence, we defined a contraction of $S$ onto $p$. 
\end{proof}

Let $W$ be a family of left-invariant vector fields on a Lie group $\G$.
We denote by $S_W$  the   semigroup generated by  $\exp(W)$. With abuse of terminology, we shall say that $S_W$ is the {\em semigroup generated by} $W$.  The set $S_W$ can be described as 
\begin{equation}
\label{def:SW}
S_W := \bigcup_{k=1}^\infty (\exp(W))^k,\end{equation}
where
\begin{equation}
\label{(exp(W))^k}
(\exp(W))^k :=
 \{ \Pi_{i=1}^k \exp(w_i)\mid   w_1,\ldots, w_k\in W  \}.\end{equation}
We stress that it may happen (see for example Section~\ref{sec:caso3}) that $S_W\neq (\exp(W))^k$, for each integer $k$.

We will prove  the following results in a larger generality. However, our main interest is when $W$ is what we shall call a {\em horizontal half-space} in a Carnot group, that is, after  
fixing a scalar product on the first stratum $V_1$ 
and given  $X\in V_1$, we consider $W=W_X$ to be the orthogonal space to $X$ within $V_1$.

\begin{remark}\label{rmk:semigroup:cone}
We observe that if $W$ is a subspace of the Lie algebra of a Carnot group and is invariant with respect to positive dilations, i.e.,  $\delta_\lambda(W)=W$, for all $\lambda>0$, as for example is the case of half-spaces in $V_1$, then the semigroup $S_W$ is a cone. Indeed, since each $\delta_\lambda$ is a group homomorphism, if $\Pi_{i=1}^k \exp(w_i)$ is an arbitrary point of $S_W$, then 
$\delta_\lambda(\Pi_{i=1}^k \exp(w_i))= \Pi_{i=1}^k \exp(\delta_\lambda(w_i))$ is also a point of $S_W$.
\end{remark}

The fact that semigroups generated by sets that are Lie bracket generating have interior is standard in geometric control theory, where such semigroups are called attainable sets. 
We refer the reader to  \cite[Chapter 8, Theorem 8.1 and Proposition 8.5]{Agrachev_Sachkov}.
 However, since the proof is short, we include a self contained one without claiming any originality. In the next statements, we shall denote by $\R_+$ the set of positive real numbers: $\R_+:= (0,+\infty)$.

\begin{proposition}\label{semigroup:has:interior}
Let $G$ be a Lie group and $W$ be a subset of its Lie algebra ${\rm Lie}(G)$ such that $\R_+ W = W$. If  there is no proper subalgebra of ${\rm Lie}(G)$ containing $W$, then the semigroup $S_W$ has interior.
\end{proposition}
\begin{proof}
Denote by $\Delta$ the left-invariant subbundle such that $\Delta_{1_G}=\Span( W)$. Hence, the assumption on $W$ rephrases as the fact that $\Delta$ is completely nonholonomic.

Take $X_1\in W$. Then $M_1:=\exp(\R_+ X_1)$ is a 1-manifold contained in $S_W$.
Unless $G$ has dimension 1, there exists $p\in M_1$ such that $\Delta_p \nsubseteq T_pM_1$.
Since $ T_pM_1$  is a vector space, we get that $({\rm d} L_p) (W) \nsubseteq T_pM_1$.
Hence there exists $X_2\in W$ such that $({\rm d} L_p) (X_2) \notin T_pM_1$.
We deduce that the map
$$(t_1, t_2) \in \R_+\times \R_+ \mapsto \exp(t_1 X_1) \exp(t_2 X_2)$$
range inside $S_W$ and it is an immersion near some point $(\bar t_1, \bar t_2)$. We have constructed a 2-manifold $M_2$ in $S_W$.

By induction, assume we have a k-manifold $M_k$ in $S_W$. If $\dim G=k$, then $M_k$ is an open set of $G$. and we are done.
If $\dim G<k$, then $M_k$ is a proper submanifold and since $\Delta$ is completely not holonomic, we find a vector $X_3\in W$ such that its associated  left invariant vector field is transverse to $M_k$. Analogously as above, flowing from $M_k$ with respect to $X_3$ at positive times we get a subset of $S_W$ that contains a (k+1)-manifolds. By induction, we conclude.
\end{proof}
\begin{lemma}\label{lem:almost regular}
If $S\subseteq \G$ is a conical semigroup with interior in a Carnot group $\G$, then $S$ is almost regular.
\end{lemma}

\begin{proof}
Fix $p\in \Int(S)$. Take an arbitrary $c\in S$.
Since $S$ is a cone, each point $\delta_t p$ is in $\Int(S)$, for all $t>0$.
Thus, since $S$ is a semigroup, each point $c\delta_t p$ is in $\Int(S)$, for all $t>0$.
Since $c\delta_t p\to c$ as $t\to 0$, then $c\in \Cl (\Int (S)).$ 
\end{proof}

With what  just observed in Remark~\ref{rmk:semigroup:cone}, Proposition~\ref{semigroup:has:interior}, and Lemma~\ref{lem:almost regular}, we immediately deduce the following result.
\begin{corollary}\label{semigroup_Canot} 
Let $\G$ be a Carnot group. 
Let $S_W$ be the subgroup generated by a horizontal half-space $W$.
Then $S_W$ is a cone with interior and it is almost regular.
\end{corollary}

The link between the semigroups that we are considering and the theory of constant normal sets is given in the following results. The first one is immediate and relates the precise monotonicity with the cone property with respect to the generated semigroup.

\begin{proposition}\label{monotone_semigroup} In a Carnot group $\G$, let     $W\subseteq {\rm Lie}(\G)$  be  such that $\R_+W=W$. 
A set $E\subseteq \G$ has the $S_W $-cone property if and only if 
 $E$ is precisely $X$-monotone,
for all $X\in W.$

\label{Remark_unioni}
In particular, we have the following consequences:
\begin{itemize}
\item[a)] The set  $S_W $ is precisely $X$-monotone,
for all $X\in W;$
\item[b)] If $A$ is an arbitrary set  and $(E_\alpha)_{\alpha\in A}$ is a family of sets  that are precisely $X$-monotone,
for all $X\in W$, then 
 $E:=\bigcup_\alpha E_\alpha$ is precisely $X$-monotone,
for all $X\in W;$
\item[c)]Every  set $E$ that is precisely $X$-monotone, for all $X\in W$,   satisfies
$$E=\bigcup_{p\in E} p S_W,$$
if either $E$ is open or if $1_\G\in S_W$. 
\end{itemize}
 \end{proposition}
Recall that in our convention $\R_+$ does not contain $0$ so the set $W$ may not contain $0$; hence, the identity element  $1_\G$ may not be in the semigroup $S_W$. 

\begin{proof}[Proof of Proposition~\ref{monotone_semigroup} and its consequences] Since $\Pi_{i=1}^k \exp(w_i)$ is an arbitrary point in $S_W$, for some $k\in \N$, and $   w_1,\ldots, w_k\in W,$ it is enough to observe that   
$    E  \cdot S_W \subseteq     E$ if and only if  $E  \exp( w ) \subseteq     E$ for all  $w\in W$. Since $\R_+W=W$, we also have that 
$E  \exp( X ) \subseteq     E$ if and only if  $E  \exp( t X ) \subseteq     E$, for all $t>0$, which is the definition of 
precisely $X$-monotone.

Regarding the consequence a), just recall Corollary~\ref{semigroup_Canot} and 
Remark~\ref{rmk:semi_coneprop}. Regarding b), 
$    E_\alpha  \cdot S_W \subseteq     E_\alpha$ for all $\alpha$ then taking unions we get
$ ( \bigcup_\alpha  E_\alpha  )\cdot S_W = \bigcup_\alpha ( E_\alpha  \cdot S_W )\subseteq   \bigcup_\alpha  E_\alpha$. Finally, c) is trivial since if $1_\G\in S_W$ then 
$E\subseteq \bigcup_{p\in E} p S_W \subseteq E.$
The case $E$ open is similar.
\end{proof}


\begin{proposition}\label{monotone_semigroup2} Let $\G$ be a Carnot group, $E\subseteq \G$ measurable, and $W\subseteq {\rm Lie}(\G)$ a subset such that $\R_+W=W$. 
If for all $X\in W$ the set $E$ is $X$-monotone,
then there exists a representative $\tilde E$ of $E$  such that $  \tilde E $ has the $S_W $-cone property (one may take the one in \eqref{def:tilde:E} with respect to any  right-invariant doubling distance $\rho$);
 
\end{proposition}
\proof  Assume $X\uno_{E}\geq 0$ for all $X
\in W$ and take $\Pi_{i=1}^k \exp(w_i)$ an arbitrary point in $S_W$, for some $k\in \N$, and $   w_1,\ldots, w_k\in W.$  Then Proposition~\ref{monotone_direction} implies
$  \tilde E^\rho  \exp(w_1)\cdots  \exp(w_k)\subseteq   \tilde E^\rho$, which is the $S_W $-cone property.
\qed

We rephrase the previous result in   case $W$ is a horizontal half-space.

\begin{corollary}\label{monotone_subgroup_Carnot} 
Let $\G$ be a Carnot group and $E\subset \G$ a measurable set. Fix a scalar product on the first stratum $V_1$   of the Lie algebra of $\G$.
Given  $X\in V_1$, set $W=W_X$ to be the orthogonal space to $X$ in $V_1$.
Let $S_W$ be the subgroup generated by $W$.
The following are equivalent:
\begin{itemize}
\item[(1)] $E$ has 
constant normal equal to $X$,
\item[(2)]    $E$ has a representative $\tilde E$ such that    $  \tilde E \cdot S_W \subseteq   \tilde E.$
\end{itemize}
\end{corollary}

\begin{remark}\label{rmk:constant_precise}
 
As a consequence of Corollary~\ref{monotone_subgroup_Carnot} and Proposition~\ref{monotone_semigroup}, we have that if a set   has constant normal, then it is  has a representative that has precisely constant normal. One such representative is 
obtained when in   \eqref{def:tilde:E} one takes  any  right-invariant doubling distance $\rho$.
 
\end{remark}

\subsection{Topological consequences of   the cone property}


In the following, for a set $E$ in a Carnot group $\G$, 
we say that  its {\em Carnot-Lebesgue representative} is the set 
given by \eqref{def:tilde:E} where one chooses as $\rho$ a Carnot (left-invariant) distance. Namely, the Carnot-Lebesgue representative of $E$ is the set
\begin{equation}
\label{def:tilde:E:Carnot}
  \tilde E =\tilde E_{\rm Carnot}:=\left\{x\in \G:
  \lim_{r \to 0^+} \frac{\vol(B_\rho(x,r) \cap E)}{\vol(B_\rho(x,r))} =1\right\}.
\end{equation}

\begin{lemma}\label{lem:CL:open}
If a set $E$ has the $C$-cone property with  respect to an open cone $C$, 
 then its Carnot-Lebesgue representative $\tilde E$ is open and has the $C$-cone property. 
\end{lemma}
\begin{proof}

We shall show that for each point $p\notin \Int(\tilde E)$ the density of $E$ at $p$ is not 1 and so $p\notin  \tilde E $. 
We first claim that $E\cap  p C^{-1} $ is empty. Indeed, by contradiction, suppose that
 $q\in E\cap   p C^{-1}$. 
 On the one hand, since $q\in E$ and $E$ has the $C$-cone property, then $q C \subseteq E$.
 On the other hand, since  $q\in   p C^{-1}$, then 
$p   \in q C \subseteq E$.
   This implies that 
$p$ is in an open set contained in $E$. Thus $p\in \Int(  E)$,  which is a contradiction with $p\notin \Int(\tilde E)$ since $\Int(  E) \subset \Int(\tilde E)$. We deduce that $E\cap p  C^{-1} $ is empty. 

As a consequence we claim  that  the density of $E$ at $p$ with respect to the Carnot distance $\rho$ is strictly less that 1. Indeed, using   left invariance and homogeneity, we have

$$ \hspace{-2.5cm}\frac{\vol(B_{\rho}(p,r) \cap E)}{\vol(B_{\rho}(p,r))} =
\frac{\vol(B_{\rho}(p,r) \cap E\cap p C^{-1})+\vol(B_{\rho}(p,r) \cap E\cap (p C^{-1})^c)}{\vol(B_{\rho}(p,r))}  
 $$

$$\leq \frac{ \vol(B_{\rho}(p,r)  \cap (p C^{-1})^c)}{\vol(B_{\rho}(p,r))} 
 = \frac{ \vol(B_{\rho}(1_\G,1)  \cap (  C^{-1})^c)}{\vol(B_{\rho}(1_\G,1))}  
 =1- \frac{\vol(B_{\rho}(1_\G,1) \cap C^{-1})  }{\vol(B_{\rho}(1_\G,1))} 
<1,$$
where in the last inequality we used that $C^{-1}$ has non-empty interior.
We showed that at every point $p\notin \Int(\tilde E)$ the density of $E$   is strictly less that 1. Therefore the set $\tilde E$ is open.

Regarding the $C$-cone property of $\tilde E$, for $p\in \tilde{E}$ let $p_j \to p$, $p_j \in E$ (such a sequence exists since the density of $E$ at $p$ is $1$).
Because of the $C$-cone property of $E$, we get that
   $p_j C \subset E$. 
Consequently, since  $p_j C$ is open and therefore in this set the density of $E$  is $1$, we have
     $p_j  C \subset \tilde{E}$. Therefore $\cup_{j\in \N} p_j C \subset \tilde{E}$; note moreover that $\cup_{j\in \N} p_j C \supset p C$ because $p C$ is open. These two inclusions show the cone property of $\tilde{E}$ with respect to $C$.
\end{proof}

For the next result recall that a distance on a topological space is called {\em admissible} if it gives the given topology. In particular, we can take Riemannian or sub-Riemannian distances on our Carnot group, without assuming a particular invariance.
The key property is that for the Lebesgue representative $\tilde E$ of a set $E$ with respect to any such a distance, one has
$\Int(E) \subseteq \tilde E$.

\begin{lemma}\label{lem:closure_cone_prop}
Let $C\subseteq \G$ be  a cone in a Carnot group $ \G$ and $E\subseteq \G$ a set with the $C$-cone property.
Let $\tilde E$ be a Lebesgue representative of $E$ with respect to some admissible distance.
If $\tilde E$ is open  and $C$ is almost regular, then $\tilde E$ has the $\bar C$-cone property.
\end{lemma}

\begin{proof}
Fix $p\in \tilde E$ and $\bar c\in \bar C$. We want to show that $p\bar c\in \Int(E)$.
Since $C$ is almost regular, there exists a sequence $c_j\in \Int(C)$ converging to $c$.
Since $\tilde E$ is open there exists an open neighborhood $U$ of $p$ contained in $\tilde E$.
Since $\tilde E$ is a Lebesgue representative of $E$, then necessarily $E\cap U$ is dense in $U$.

Notice that since $c_j\to \bar c$ then $p\bar c c_j^{-1}\to p$. Hence, 
since $U $ is a neighborhood of $p $, we have that for $j$ large enough
$p\bar c c_j^{-1}$ is in $U$; equivalently, we have that 
$U c_j$ is a neighborhood of $p\bar c$. Fix one such $j$ and take an open neighborhood $V$ of $c_j$ such that 
$V\subseteq C$ and $UV$ is a neighborhood of $p\bar c$.

We now focus on the set $UV$, which we claim to be equal to $(U\cap E)V$. Regarding the nontrivial containment of this equivalence, take $uv$ with $u\in U$ and $v\in V.$ Since $V$ is open and $U\cap E$ is dense in $U$ there exists $u'\in U\cap E$ close to $u$ such that $(u')^{-1}uv \in V$. Thus
 for some  $v'\in V$, we have $uv =u'v'\in UV$. Hence, the equivalence $UV=(U\cap E)V$ is proved.
 
 Since $E$ has the $C$-cone property and $V\subseteq C$, we that that
 $$UV=(U\cap E)V \subseteq E. $$
 We infer that $UV$ is an open set containing  $p\bar c$ that is contained in $E$. Therefore  $p\bar c \in \Int(E)$. Consequently,  we come to the conclusion that 
$p\bar c \in  \tilde  E$. 
\end{proof}

\begin{lemma}\label{lem:reg:open}
 
Let $C\subseteq \G$ be  a   cone  in a Carnot group $ \G$ and $E\subseteq \G$ a set with the $C$-cone property.
 If $  E$ and $C$ are open, then $  E$ is regularly open.
\end{lemma}
\begin{proof}
Since $E$ is open we have $ E \subseteq \Int ( \Cl (E) )$.  To prove the opposite inclusion, take $p\in \Int ( \Cl (E)) $. Then there exists an neighborhood $U$ of $p$ such that $U\subset \Cl(E)$. Since $C$ is open, so is $C^{-1}$, and since $E$ is dense in $U$,  we can take $q\in U\cap pC^{-1} \cap E$. By the $C$-cone property we conclude that
$$p\in q C \subset E.$$
Therefore, we get $ E = \Int ( \Cl (E)) $.
\end{proof}

\subsection{Proofs of Theorems~\ref{main:thm}, \ref{thm1}, and \ref{thm2}}
Here we complete the proofs of Theorems~\ref{thm1}, \ref{thm2}, and \ref{main:thm}, respectively.

\begin{proof}[Proof of Theorem~\ref{thm1}]
We claim that the proof follows from Proposition~\ref{monotone_semigroup}. Indeed, if 
$E$ is a precisely constant-normal set with respect to some closed half-space $W\subset V_1$, then Proposition~\ref{monotone_semigroup} implies that $E$ has the cone property with respect to the semigroup $S_W$ generated by $W$. From Corollary~\ref{semigroup_Canot} we have that $S_W$ is a cone with interior. In Remarks~\ref{cone_prop_open} and \ref{rmk:int:semigroup}, we observed  that interior of a conical semigroup with interior is an open conical semigroup. Obviously we also have that $E$ has the cone property with respect to the interior $\Int(S_W)$.
Being $\Int(S_W)$ a conical semigroup, from Lemma~\ref{lem:contractible} we have that it is contractible.
\end{proof}

\begin{proof}[Proof of Theorem~\ref{thm2}]
By Lemma~\ref{lem:CL:open} we have that if a set $E$ has the  cone property with  respect to an open cone $C$, 
 then its Carnot-Lebesgue representative $\tilde E$ is open and has the $C$-cone property. 
 From Lemma~\ref{lem:reg:open} we have that
$\tilde E$ is regularly open, since $\tilde E$ and $C$ are open.
 In Remark~\ref{cone_prop_complement} we noticed that consequently we have that the complement  $\tilde E^c$ of $\tilde  E$ has the cone property with respect to $C^{-1}$, which is also an open cone.
 From Proposition~\ref{prop:contractible} we infer that both $\tilde E$ and $\tilde E^c$ are contractible.
\end{proof}

We shall next prove a stronger version of Theorem~\ref{main:thm}.
\begin{theorem}
\label{main:thm2}
 Let $E$ be a subset in a Carnot group that has constant normal   with respect to a horizontal half-space $W$. Let   $S_W $ be the semigroup generated by $W$, which is a conical semigroup with interior.
  Then the Carnot-Lebesgue representative  of $E$ has the $\Cl(S_W) $-cone property, 
  is regularly open, is contractible, and its complement is contractible.
\end{theorem}

\begin{proof}[Proof of Theorem~\ref{main:thm2} and hence of Theorem~\ref{main:thm}]
In this proof we shall consider two representatives of the  given   set $E$. 
When applying Proposition~\ref{monotone_semigroup2} we shall take a
representative $\tilde E_{\rm right}$ of $E$  
obtained when in   \eqref{def:tilde:E} one takes  any  right-invariant doubling distance $\rho$; 
successively, we shall consider the Carnot-Lebesgue representative $\tilde E_{\rm Carnot}$ as in \eqref{def:tilde:E:Carnot}.

From Remark~\ref{rmk:constant_precise} (see also Proposition~\ref{monotone_semigroup2} and Corollary~\ref{monotone_subgroup_Carnot}) we have that, since  $E$ has constant normal, then  
   $  \tilde E_{\rm right} $ has the cone property  with respect to  the   semigroup  $S_W $ generated by some half-space $W\subset V_1$. Let $C_W:=\Int(S_W)$. From Corollary~\ref{semigroup_Canot}, and Remarks~\ref{cone_prop_open} and \ref{rmk:int:semigroup}, recall that we have that both $S_W $ and $C_W $
     are conical semigroups and they have interior, that is $C_W\neq \emptyset$.  
From Lemma~\ref{lem:CL:open} we deduce that, since $  \tilde E_{\rm right} $ has the cone property  with respect to  the    open cone $C_W$,  then the Carnot-Lebesgue representative $\tilde E_{\rm Carnot}$ is open and has the $C_W$-cone property. 
In Lemma~\ref{lem:almost regular}, we proved that both $S_W $ and $C_W $ are
almost regular.
 We claim that $\tilde E_{\rm Carnot}$ has the cone property with respect to the closure of $S_W$.
Indeed,
 from Lemma~\ref{lem:closure_cone_prop} we have that $\tilde E_{\rm Carnot}$ has the cone property with respect to the closure of $C_W$.
 Since $S_W$ is almost regular, then 
 $$  \Cl(S_W) = \Cl (\Int (S_W)) =\Cl  (C_W)   .$$
Then $\tilde E_{\rm Carnot}$ has the $\Cl(S_W)$-cone property.
Applying Theorem~\ref{thm2} with the fact that  $\tilde E_{\rm Carnot}$ has the $C_W$-cone property, we get the rest of the claimed properties.
 \end{proof}

\section{Consequences on density and boundaries}
Let $\G$ be a Carnot group with a fixed Haar measure $\vol$ and a fixed Carnot distance $\rho$. As it is standard in Geometry Measure Theory, given a measurable set $E\subseteq \G$ we define the lower density and  upper density of $E$ at $x$, respectively as
\begin{equation}
\label{def:density}
 \Theta_*(x,E):=
  \liminf_{r \to 0^+} \frac{\vol(B_\rho(x,r) \cap E)}{\vol(B_\rho(x,r))}   
  \qquad\text{and}\qquad
 \Theta^*(x,E):=
  \limsup_{r \to 0^+} \frac{\vol(B_\rho(x,r) \cap E)}{\vol(B_\rho(x,r))} 
  .\end{equation}

If $E\subseteq\G$, as it is usual in the literature, in Section~\ref{sec:boundaries} we define the {\em measure theoretic boundary} 
$\partial_{\rm mt}E$,
also called  {\em essential boundary}, and the {\em   De Giorgi's reduced boundary} 
$\partial_{\rm DG}E$,
also denoted by $\mathscr F E$. 

The aim of this section is to show that if $E$ is the Carnot-Lebesgue representative of a constant normal set, then the are global density estimates (see Proposition~\ref{prop:densities}) and these boundaries coincide with the topological boundary (see Proposition~\ref{prop:boundaries:coincide}).

\subsection{Various kinds of boundaries}\label{sec:boundaries}
\begin{definition}[$\partial_{\rm mt}E$] \label{mt_boundary}
 The {\em measure theoretic boundary} of $E$ is the set of points where the volume density of $E$ is neither 0 nor 1. Namely, 
 $$\partial_{\rm mt}E := \{ x\in\G :  \Theta_*(x,E)\neq 1 \text{ or }  \Theta^*(x,E)\neq0\}.$$
 \end{definition}

 In what follows, given a measurable subsets $E$ of a Carnot group $\G$, we denote by $\uno_E$ its characteristic function, which is in the space $L^1_{\rm
loc}(\G)$ of locally integrable functions. We denote by ${\mathcal M}(\G)$ the space of Radon real-valued measures on $\G$. 
 
 \begin{definition}[Sets of locally finite perimeter]
 
A Borel subsets $E\subset\G$  of a Carnot group $\G$ has \emph{locally finite perimeter} if
$X\uno_E\in{\mathcal M}(\G)$ is a Radon measure for any $X\in V_1$.
Fixing a basis $X_1, \ldots, X_m$ of $V_1$
we can define the
$\R^m$-valued Radon measure
\begin{equation}\label{dchiE}
D\uno_E:=(X_1\uno_E ,\ldots,X_m\uno_E).
\end{equation}
We call the total variation\footnote{  Recall that the
\emph{total variation}  $|\mu|$ of an $\R^m$-valued measure $\mu=(\mu_1,\ldots,\mu_m)$ with $\mu_i\in{\mathcal M}(\G)$
is the smallest nonnegative measure $\nu$ defined on Borel sets of
$\G$ such that $\nu(B)\geq |\mu(B)|$ for all bounded Borel set $B$;
it can be explicitly defined by
$$
|\mu|(B):=\sup\left\{\sum_{i=1}^\infty |\mu(B_i)|:\ \text{$(B_i)$
Borel partition of $B$, $B_i$ bounded}\right\}.$$}
$|D\uno_E|$ of $D\uno_E$ the {\em perimeter measure} of $E$.
\end{definition}

 \begin{definition}[$\partial_{\rm DG}E$] \label{DG_boundary}
Let $E\subseteq \G$ be a set of locally finite perimeter of a Carnot group $\G$. We define the {\em De Giorgi's reduced boundary}
  $\partial_{\rm DG}E$ of $E$ as the set of points $x\in {\rm supp\,}|D\uno_E|$
where:
\begin{itemize}
\item[(i)] the limit $\nu_E(x)=(\nu_{E,1}(x),\ldots,\nu_{E,m}(x)):=\displaystyle{
\lim\limits_{r\downarrow 0}\frac{D\uno_E(B_r(x))}{|D\uno_E|(B_r(x))}}$ exists;
\item[(ii)] $|\nu_E(x)|=1$.
\end{itemize}
\end{definition}
 
 \begin{remark}\label{remark:normalratioconstant}
 If $E\subset \mathbb{G}$ is a constant-normal set with normal $X$, then $\frac{D\uno_E(B_r(x))}{|D\uno_E|(B_r(x))}$ is equal to $X$ for every $x\in {\rm supp\,}|D\uno_E|$ and for every $r$. In particular we have, for every $x\in {\rm supp\,}|D\uno_E|$, that $\nu_E$ exists (equal to $X$) and $|\nu_E(x)|=1$. 
  \end{remark}

 \begin{proposition}\label{prop:boundaries:coincide}
 Let $E\subset\G$ be  measurable subset of a Carnot group $\G$.
 \begin{enumerate}
 \item we have $
\partial_{\rm mt}E\subset\partial E   \qquad\text{and}\qquad \partial_{\rm DG}E\subset\partial E$.
\item If $E$ has the $C$-cone property with $\Int(C)\neq\emptyset$, then  $\partial E =\partial_{\rm mt}E$.
\item If $E$ is the Carnot-Lebesgue representative (as in \eqref{def:tilde:E:Carnot}) of a  constant normal set, then  $\partial E= \partial_{\rm DG}E =\partial_{\rm mt}E = {\rm supp\,}|D\uno_E|$. 
 \end{enumerate}

 \end{proposition}
\begin{proof}
Part 1 is straightforward. Regarding part 2, the $C$-cone property immediately gives the missing inclusion $\partial E \subset\partial_{\rm mt}E$.

To prove part 3, we recall that by Theorem~\ref{main:thm2} this representative has a cone property, hence by part 2 we have $\partial E =\partial_{\rm mt}E$. Regarding De Giorgi's reduced boundary, 
%
%
recall that (Definition~\ref{DG_boundary}) for an arbitrary set $E$ we have $\partial_{\rm DG}E\subseteq{\rm supp\,}|D\uno_E|$. For sets of constant normal, the reverse inclusion is 
given by Remark~\ref{remark:normalratioconstant}. Hence 
$\partial_{\rm DG}E={\rm supp\,}|D\uno_E|$. 
 Together with parts 1 and 2 this gives ${\rm supp\,}|D\uno_E|=\partial_{\rm DG}E\subset\partial E = \partial_{\rm mt}E$. For the reverse inclusion let $x \in \partial_{\rm mt}E$, then in any neighbourhood $U$ of $x$ both $E \cap U$ and $U\setminus E$ have positive measure. This implies that $E$ must have positive perimeter in $U$. The arbitrariness of $U$ gives $x\in {\rm supp\,}|D\uno_E|$, so ${\rm supp\,}|D\uno_E|\supset \partial_{\rm mt}E$ and part 3 is proved.
\end{proof}

\subsection{Global density estimates}

In this subsection we deduce density estimates that are know to hold for locally for sets of finite perimeter. In fact, we get global estimates, i.e., for all radii, for constant-normal sets. 

\begin{proposition}\label{prop:densities}
	If a subset $ E\subset \G $ of a Carnot group $\G$ has $C$-cone property with $\Int(C)\neq\emptyset$, then there is a constant 
	$ 0<l_C \leq 1/2$ such that for all $x\in \partial E$ we have 
	\[
	l_C\vol(B_\rho(x,r))\leq \vol(B_\rho(x,r) \cap E) \leq (1-l_C)\vol(B_\rho(x,r)) \quad \forall r \in (0,\infty).
	\]
	In particular, for all $x\in\partial_{\rm mt}E$ the density is pinched, in the sense that
	$$0< l_C\leq \Theta_*(x,E)\leq
 \Theta^*(x,E)  \leq (1-l_C) <1.$$
	Moreover, if $ E\subset \G $ has constant normal, then $ l_C$   can be chosen to 
	depend only on $ \G $. 
\end{proposition}
\begin{proof}
Regarding the existence of $l_C$, consider the following quotients and use the fact that if $x\in E $ then $xC\subseteq E$ and get
$$\frac{\vol(B_\rho(x,r) \cap E)}{\vol(B_\rho(x,r))} 
\geq
\frac{\vol(B_\rho(x,r) \cap xC)}{\vol(B_\rho(x,r))} 
= \frac{\vol(B_\rho(1_\G,r) \cap C)}{\vol(B_\rho(1_\G,r))} 
= \frac{\vol(B_\rho(1_\G,1) \cap C)}{\vol(B_\rho(1_\G,1))} =:l_C,$$
where we used left invariance and that $C$ is a cone. Since $C$ has   interior we get $l_C>0$. The existence of the upper bound is analogue thanks to Remark~\ref{cone_prop_complement}.
The consequence about the pinched densities is immediate from the definition of densities.
Regarding the last statement, it is enough to show that 
\begin{equation}\label{cribbio}
0< \inf\{ l_{S_W}\; :\; W\subset V_1 \text{ horizontal half-space}\}.
\end{equation}
To prove this claim we shall use the notion of free Carnot groups, see \cite[p.\ 45]{varsalcou} or \cite[p.\ 174]{Varadarajan:1984:Lie_groups}.
If the Carnot group $\G$ has rank $m$ and step $s$ then there is a surjective Carnot morphism $\pi:\mathbb F \to \G$ between the  
free Carnot group  $\mathbb F=\mathbb F_{m,s}$ of rank $m$ and step $s$ onto $\G$. 
Moreover, we equip $\mathbb F$ with a Carnot distance that makes $\pi$ a submetry: for all $p\in \mathbb F$ and all $r>0$, we have
\begin{equation}\label{submetry}
\pi(B_{\mathbb F}(p,r) ) = B_{\G}(p,r) := B_{\rho}(p,r).
\end{equation}
Since $\mathbb F$ is a free Carnot group, then the action of $GL(V_1^{\mathbb F})$ on the first layer $V_1^{\mathbb F}$ of $\mathbb F$ extends to an action by Carnot morphisms of $\mathbb F$. 
Moreover, fixing coordinates on $V_1^{\mathbb F}$ so that it becomes isometric to $\R^m$,  
the orthogonal group $O(m)$ acts by isometries of $\mathbb F$ and acts transitively on the space of horizontal half-spaces.
%
Namely, fixed a horizontal half-space $W_0$, we have that for every horizontal half-space $W$ there exists $A\in O(m)$ such that $A(W_0)=W$ and therefore $A(S_{W_0}^\mathbb F)=S_W^\mathbb F$, where the latter ones are semigroups generated in $\mathbb F$. Therefore, we get 
\begin{eqnarray*}
l_{S_W} {\vol(B_\G(1_\G,1))}  &=& {\vol(B_\G(1_\G,1) \cap S_W)}
\\&=& 
\vol(\pi (B_\mathbb F(1_\mathbb F,1) ) \cap \pi (S_W^\mathbb F))
\\&\geq&
\vol(\pi (B_\mathbb F(1_\mathbb F,1) \cap  S_W^\mathbb F))
\\&=&
\vol(\pi (B_\mathbb F(1_\mathbb F,1) \cap  A (S_{W_0}^\mathbb F))),
\end{eqnarray*}
where in the second equality we used \eqref{submetry} and that $\pi (S_W^\mathbb F)= S_W $, by identifying via $\pi$ the horizontal spaces.
On the one hand, the quantity $\vol(\pi (B_\mathbb F(1_\mathbb F,1) \cap  A (S_{W_0}^\mathbb F)))$ when $A$ varies in the compact set $O(m)$ must have a minimum. On the other hand, this minimum, say realized by some $A'$, cannot be zero since
the horizontal 
half-space $W':=A'(W_0)$
generates a semigroup 
with interior; say there is a ball $B(p,r) \subseteq B_\mathbb F(1_\mathbb F,1) \cap  S_{W'}^\mathbb F$, for some $p\in \mathbb F$ and  $r>0$. Therefore, we get
\begin{eqnarray*}\vol(\pi (B_\mathbb F(1_\mathbb F,1) \cap  A' (S_{W_0}^\mathbb F))) &=& 
\vol(\pi (B_\mathbb F(1_\mathbb F,1) \cap  S_{W'}^\mathbb F)\\
&\geq &\vol(\pi (B_{\mathbb F} (p,r)    ))
\\&=&\vol(B_{\G} (\pi( p),r)    )>0 .
\end{eqnarray*}
In conclusion, we proved \eqref{cribbio}.

Regarding the upper bound with the term $(1-l_c)$, we pass to the complement of $E$ and  refer to Remark~\ref{cone_prop_complement}.
\end{proof}

\begin{proposition}
For every  Carnot group $\G$ there exist positive constants $ k_\G ,K_\G$ such that 
if $ E\subset \G $ is a  constant normal set, then for all $x\in\partial_{\rm mt}E$ we have
 
	\[
	k_\G \frac{\vol(B_\rho(x,r))}{r} \leq |D\uno_E|(B_\rho(x,r)) \leq  K_\G\frac{\vol(B_\rho(x,r))}{r} ,\quad \forall r \in (0,\infty).
	\]
	\end{proposition}

\begin{proof}

 The Poincar\'{e} inequality is true in Carnot groups, so \cite[Remark 3.4]{Amb02} and Proposition~\ref{prop:densities} give (here $Q$ denotes the Hausdorff dimension of $\mathbb{G}$)
 $$|D\uno_E|(B_\rho(x,r)) \geq C_{\mathbb{G}} \min\left\{\left(\vol(E\cap B_\rho(x,r)) \right)^{\frac{Q-1}{Q}},\left(\vol(E\cap B_\rho(x,r)) \right)^{\frac{Q-1}{Q}}\right\}  \geq C_{\mathbb{G}} l_C^{\frac{Q-1}{Q}} r^{Q-1}$$
for every $r>0$ and for all $x\in\partial_{\rm mt}E$.  

 For the upper bound, we note that we can repeat the proof of \cite[Lemma 2.31]{fssc} with $r_0=\infty$ and for $x$ arbitrary in $\partial_{\rm mt}E$. Indeed, the choice of $r_0$ in \cite{fssc} is needed to ensure \cite[(2.32)]{fssc}; in our case, Remark~\ref{remark:normalratioconstant} guarantees \cite[(2.32)]{fssc} for all $r$ and for all $x\in\partial_{\rm mt}E$ (more precisely, it gives that \cite[(2.32)]{fssc} is an equality, without the factor $2$ on the right-hand-side). We then conclude (using Proposition~\ref{prop:densities} for the second inequality)
 $$|D\uno_E|(B_\rho(x,r)) \leq\frac{\vol_\rho(E \cap B_\rho(x,2r))}{r}\leq (1-l_C) 2^Q \frac{\vol_\rho(B_\rho(x,r))}{r}.$$
 \end{proof}

\section{Euclidean cones and wedges of semigroups}

In this section, with the aim of writing precisely constant normal sets as H\"older upper-graphs, we consider the largest cone inside the semigroup generated by a horizontal half-space. We shall show that such a cone shares several properties, which are referred to as being a Lie wedge. Consequently, we deduce that there are several half one-parameter subgroups in the semigroup.
Our viewpoint is highly inspired by \cite[Chapters 1-3]{Hilgert_Neeb:book_semigroup}.

Let $\G$ be a Carnot group. Recall that the exponential map $\exp$ from the Lie algebra $\mathfrak g$ of $\G$ to $\G$ is a global diffeomorphism, whose inverse we denote by $\log$.

\begin{remark}\label{rmk:ad_inv}
We claim that if $S\subseteq \G$ is a semigroup, then 
\begin{equation}\label{Levico2019}
 e^{\ad_X} \log(S) = \log(S), \qquad \text{ for all } X \text{ such that } \pm X\in \log(S).
 \end{equation}
Indeed, we have
\begin{eqnarray*}
 e^{\ad_X} \log(S) =\Ad_{\exp(X)} \log(S) =\log C_{\exp(X)} (S)
=\log ({\exp(X)} S\exp(-X))\subseteq  \log(S),
\end{eqnarray*}
where in the last containment we used that, by assumption, the set $S   $ is a semigroup and $\exp(\pm X)\in S$.
Since the inverse map of $ e^{\ad_X}$ is $ e^{\ad_{-X}}$ then by symmetry we also have the other inclusion in \eqref{Levico2019}.
\end{remark}

In particular, if $S   $ is the  semigroup generated by a horizontal half-space $W$, then for the set $\mathfrak s:=
\log(\bar S)$, where $\bar S=\Cl(S))$, we have
$$ e^{\ad_X} \mathfrak s = \mathfrak s, \qquad \text{ for all }     X\in W\cap(-W).$$ 
Notice that $W\cap(-W)$ is the hyperplane in $V_1$ that is the boundary of $W$ within $V_1$.

In the following discussions, a subset $\mathfrak c$ of the Lie algebra $\mathfrak g$ is said to be a {\em cone} (or, more precisely, an {\em Euclidean cone}) if
$$t \mathfrak c = \mathfrak c   , \qquad \text{ for all } t>0.    $$
This notion should not be confused with the one \eqref{cone_def} of cones as subsets of the Carnot group $\G$. 

\begin{definition}\label{def:wedge}
To every semigroup $S\subseteq \G$ we associate the set
\begin{equation}\label{wedge}
 \mathfrak w _S:= \{X\in \mathfrak g \; : \; \R_+ X \subseteq \log(  S) \}.
 \end{equation}
 We shall refer to this set as the {\em wedge tangent to $S$}. 
\end{definition}
Actually, the notion of   (Lie) wedge is present in the literature, see \cite[Section 1.4]{Hilgert_Neeb:book_semigroup}, and agrees with the one of this paper. For an abstract viewpoint, we recall here the standard notion. Let $\mathfrak g$ be a Lie algebra. A subset $\mathfrak w\subseteq \mathfrak g$  is said to be a {\em  Lie wedge} if
\begin{itemize}
\item $\mathfrak w$ is a closed convex cone
\item $e^{{\rm ad} X }\mathfrak w=\mathfrak w$, for all $X\in \mathfrak w\cap (-\mathfrak w).$
\end{itemize}
Next lemma clarifies that the  wedge $ \mathfrak w _S$ tangent to a closed semigroup $S$ is a Lie wedge.
The proof is not original, see for example \cite[Proposition 1.4]{Hilgert_Neeb:book_semigroup}.
\begin{lemma}\label{lem:CH}
 If $S\subseteq \G$ is a closed semigroup, 
 then  the  wedge $ \mathfrak w _S$ tangent to $S$, defined in \eqref{wedge}, satisfies the following properties:
 \begin{enumerate}
\item $ \mathfrak w _S$ is the largest cone in $\log(  S)$;
\item $ \mathfrak w _S$ is closed and  convex; 
\item $ \mathfrak w _S$ is invariant under $ e^{\ad_X} $ for each $X\in \log(  S)\cap(-\log(  S))$, i.e.,
\begin{equation}\label{ad_inv}
 e^{\ad_X}  \mathfrak w _S =  \mathfrak w _S, \qquad \text{ for all } X \text{ such that } \pm X\in \log(  S).
 \end{equation}
\end{enumerate}
\end{lemma}
\begin{proof}
Set $\mathfrak s:=
\log(  S)$. 
By construction, an element $X$ is in $ \mathfrak w _S$ if and only if 
$ \R_+ X \subseteq \mathfrak s$. Thus $ \mathfrak w _S$ is the largest cone in $\mathfrak s$.
Since $\mathfrak s$ is closed, then the closure of $ \mathfrak w _S$ is a cone in $\mathfrak s$. By maximality of 
 $ \mathfrak w _S$, we deduce that it is closed.
 
To check that $ \mathfrak w _S$ is convex, since
$\mathfrak w _S  $ is a cone, it is enough to show that for all $X, Y\in \mathfrak w _S $ we have $X+ Y\in \mathfrak w _S $.
Recall the formula, which holds in all Lie groups, 
 \begin{equation}
\label{sum}
\exp(X+Y) = \lim_{n\to \infty} \left(\exp\left(\tfrac1n X\right) \exp\left(\tfrac1n Y\right) \right)^n.
 \end{equation}
Since  $\R_+ X, \R_+ Y \subseteq \log(  S)$, then 
$\exp(\frac1n X), \exp(\frac1n Y) \in   S$, for all $n\in \N$.
Consequently, since $  S$ is a semigroup, we have $ \left(\exp(\frac1n X) \exp(\frac1n Y) \right)^n \in   S$.
Being $  S$ closed by assumption, we get from \eqref{sum} that $\exp(X+Y)  \in   S$.

Regarding \eqref{ad_inv} take $X $ such that $\pm X\in \mathfrak s$, so by Remark~\ref{rmk:ad_inv} we have 
$ e^{\ad_X}\mathfrak s =\mathfrak s$.
Consequently, since the map  $e^{\ad_X}$ is linear, it sends cones in $\mathfrak s$ to cones in $\mathfrak s$. In particular, this map fixes the largest cone, i.e., it fixes $ \mathfrak w _S$. 
\end{proof}

\begin{remark}
If $S$ is a semigroup and $\pm X,Y\in \log(S)$, then we have that 
\begin{equation}\label{rem:conj}
\exp(\Ad_{\exp(X)}Y) = \exp((C_{\exp(X)})_*Y)  = C_{\exp(X)}(\exp(Y))=
  {\exp(X)}\exp(Y)\exp(-X)\in S,
\end{equation}
where we have used that $
\Ad_g$ is by definition the differential of $C_g$ and that exp intertwines  this differential with $C_g$ and finally that $S$ is a semigroup.
\end{remark}

Every horizontal half-space $W$ is a cone inside the semigroup $S_W   $  that it generates. Therefore, the set $W$ is a subset of the
wedge $\mathfrak w _{\bar S_W}$ tangent to $\bar S_W   $.
The next result infers how one can gets the knowledge of the presence of more elements in this wedge. 

\begin{proposition}\label{prop more elements in wedge} Let $S_W   $ be the  semigroup generated by a horizontal half-space $W$.   
If $\pm X \in   \log(S_W)$ and 
$\R_+ Y  \subset \log(S_W)$, then
$$\Ad_{\exp(X)}Y \in  \mathfrak w _{\bar S_W}.$$
In particular, we have
\begin{equation}\label{}
\Ad_{\exp(X)}Y \in  \mathfrak w _{\bar S_W},\qquad \forall X\in  W\cap(-W) 
,\forall Y\in W. 
\end{equation}
\end{proposition}

\begin{proof}
By assumption, for all $t>0$ we have that $\pm X,t Y\in   \log(S_W)$. Consequently, by \eqref{rem:conj} we get $\exp(t \Ad_{\exp(X)}Y) =\exp(\Ad_{\exp(X)}tY)  \in S$. The second part is immediate since $W\subset \log(S_W)$ and $W$ is a cone.
\end{proof}
We have evidence that in groups of step $\leq 4$, the intersection $ \exp(\mathfrak w _{\bar S_W} )\cap \Int(  S_W )$ is not empty; see Remark~\ref{rem:evidence}. We wonder if this is the case in general.

\begin{conjecture} \label{conj_intersect} For every horizontal half-space $W$ in every Carnot group $\G$,  we have 
$$ \exp(\mathfrak w _{\bar S_W} )\cap \Int(  S_W ) \neq \emptyset .$$

\end{conjecture}
Recall that on the one hand, we have $\Int( S_W )=\Int( \bar S_W ) \neq \emptyset$ and that this interior is meant within $\G$. On the other hand, we have
$\exp(\mathfrak w _{\bar S_W} )\supseteq \exp(W)  \neq \emptyset$. 
 However, in some Carnot groups $\exp(W) \cap \Int(  S_W ) = \emptyset $, for example in the group considered in the next section.


\begin{proposition}
If there exists $Z\in \mathfrak w _{\bar S_W} \cap \Int(  S_W ) $, 
 then every set that has precisely constant normal with respect to $W$ is a  upper-graph in   the $Z$-direction with respect to a H\"older function (in Euclidean coordinates).
\end{proposition}

\begin{proof} 
If $p$ is a point of the boundary of such a precisely constant normal set $E$, then the cone $pS_W$ is in the set. By assumption, the half line $p\exp(\R_+ Z)$ does not intersect $\partial E$. Hence, each leaf of the foliation $\{ p\exp(\R  Z) \}_{p\in \G}$ meets $\partial E$ in one point, that is, $\partial E$
is a graph  in   the $Z$-direction and $E$ is the upper-graph.
Moreover,   the half line $p\exp(\R_+ Z)$ belongs to the interior of $S_W$ and the semigroup $S_W$ is a cone with respect to the intrinsic dilation. Since every such cone is an upper-graph in the $Z$-direction of a    H\"older function (in Euclidean coordinates), then the same conclusion must hold for the graphing function of $E$.
\end{proof}
In fact, the graphing function of the above proposition satisfies a Lipschitz intrinsic condition, similarly as in \cite{Franchi_Serapioni_2016}. However, it is important to stress that in our case the graphing direction may not be a horizontal one.

\section{The free Carnot group of rank $2$ and step $3$}
  
In this section we shall focus our attention on a specific Carnot group where many pathologies appear. Such a group is denoted by $\mathbb{F}_{23}$ and is called the {\em free Carnot group of rank $2$ and step $3$}, because any other Carnot group of  rank $2$ and step $3$ is a quotient of it.

The Lie algebra of $\mathbb{F}_{23}$ is 5 dimensional and it is generated by two vectors, which we call $X_1$ and $X_2$. A basis of the Lie algebra is completed to $X_1,\ldots, X_5$ for which the only non trivial bracket relations are:
$$X_3=[X_2,X_1],\;
X_4=[X_3,X_1],\;
X_5=[X_3,X_2].$$

The Lie group $\mathbb{F}_{23}$ is the simply connected Lie group with such a Lie algebra. Moreover, the Lie algebra has a natural stratification where the first layer is $V_1:=\Span\{X_1, X_2\}$. See \cite{LeDonne:Carnot} for an introduction to Carnot groups and stratifications.

Since $\mathbb{F}_{23}$ is a free Carnot group of rank 2 then there is an action of the general linear group GL$(2\R)$ by Lie automorphisms induced by the standard action of 
GL$(2\R)$ on
$\Span\{X_1, X_2\}$. Consequently, any pair of linearly independent horizontal vectors $X,Y\in V_1$ are equivalent to $X_1, X_2$ up to an automorphisms that preserves the stratification. For our purposes, this means that it is not restrictive to only study sets that have constant normal equal to $X_2$.
  
  We shall work in some specific coordinate system for  $\mathbb{F}_{23}$. These are called {\em exponential coordinates of the second kind} with respect to the chosen bases. Namely, an arbitrary point of  $\mathbb{F}_{23}$ is uniquely represented as
$
\exp({x}_5 X_5)\exp({x}_4 X_4)\cdot\ldots\cdot\exp({x}_1 X_1)
$
for $(x_1, \ldots, x_5)\in \R^5$.
We use such identification of  $\mathbb{F}_{23} $ with $ \R^5$. 
Using the Baker-Campbell-Hausdorff formula one computes (see Section~\ref{sec:coords}) the group law:

\begin{equation}\label{gp_law}L_x(y)=x\cdot y=
(x_1+y_1, x_2+y_2,
x_3+y_3 -x_1y_2, \hspace{4cm}\end{equation}
$$
\hspace{5cm}x_4+y_4 - x_1 y_3 +\frac{1}{2} x_1^2y_2,
x_5+y_5+x_1x_2y_2+\frac{1}{2}x_1y_2^2-x_2y_3).
$$
 
    The left-invariant vector fields $X_{1},X_{2}$ in these coordinates are
\begin{equation} \label{LIVF_F23}X_{1} = \partial_{1}\qquad \text{ and  } \qquad
X_{2}  =\partial_{2} -x_{1}\partial_{3} + \frac{x_{1}^{2}}{2}\partial_{4} + x_{1}x_{2}\partial_{5}.
\end{equation}
The other elements of the basis will not be needed. For completeness, we say that they have   the following form: 
$$
X_{3}  =  \partial_3-  x_{1}\partial_{4} -x_{2}\partial_{5} ,
\qquad X_{4}=\partial_4, 
\qquad X_{5}=\partial_5 .$$

\subsection{Semigroups generated}
In this section we shall study the semigroup $S_W$, as defined in \eqref{def:SW}, for the following three options for $W$:
$$W_1:= \{a X_1 +b X_2 :a\in \R, b\geq 0\}, \quad
W_2:= \{a X_1 +b X_2 :a\in \R, b> 0\}, $$
\begin{equation}\label{Wi}
 \qquad \text{ and  } \qquad
W_3 := \{a X_1 :a\in \R \} \cup \{  b X_2 : b> 0\}, 
.\end{equation}
Namely, the first is a closed horizontal half-space, the second one is an open horizontal half-space, and the third is the union of a line and a semi-line.
The properties that we shall prove are summarized in the following list:
\begin{itemize}
\item[i)] Each among $S_{W_1}$, $S_{W_2}$, and $S_{W_3}$ is not open and not closed. Indeed, each one does not contain $\exp(  [X_1,X_2])$ which is  in the closure, and it contains $\exp(\R_+ X_2)$  but not in the interior. Actually, we have 
\begin{equation}\label{No horizontal vectors in the interior}
\exp (V_1) \cap {\Int}(S_{W_i})=\emptyset. \end{equation}

\item[ii)] No point $\exp(X)$ with $X\in V_1$ is in their interior.
\item[iii)] With the notation \eqref{(exp(W))^k}, we have that  for $k>6$ we have 
$S_{W_1}=(\exp(W_1))^k$ and 
$S_{W_2}:=(\exp(W_2))^k$. Nonetheless, there is no $k\in \N$ such that $S_{W_3}=(\exp(W_3))^k$
\item[iv)] For $k=3$ we have  that $(\exp(W_1))^k$ has interior but it is not equal to
$S_{W_1}$. Still 
$\Cl(S_{W_1})=\Cl((\exp(W_1))^3)$. 
\end{itemize}

We shall give a precise expression for $S_{W_1}$. However, a consequence is the following containment.
Define the polynomial
\begin{equation}
\label{poly_F23}	P(x):=x_2^3x_4 -  2x_2^2 x_3^2 - 6 x_2x_3 x_5 - 6 x_5^2,\end{equation}
which is homogeneous of degree 6 with respect to the intrinsic dilations.
The semigroups  satisfies
$$ \{P(x) >0, x_2>0 \} \subseteq S_{W_2} \subseteq S_{W_1} \subseteq \{P(x) \geq0, x_2\geq 0 \} .$$
This containment will show that there are no horizontal vectors in the interior of the semigroups, i.e, \eqref{No horizontal vectors in the interior} holds.
In particular, the set $ \{P(x) >0, x_2>0 \} $ provides an example of a precisely constant normal  set that is not a continuous graph in any horizontal direction.

We also point out that the closure of the three sets $S_{W_i}$, $i=1,2,3$, is the same. Indeed, this is due to the fact that any point in $S_{W_1}$ is obtained as the end point of a piecewise linear curve where each piece has derivative in $W_1$. However, every such piece can be approximated (uniformly) by a piecewise linear curve with derivatives in $W_3$. Hence, the closure of $S_{W_3}$ contains $S_{W_1}$. Since $S_{W_3}\subseteq  S_{W_2}\subseteq  S_{W_1}$, we get
$\Cl(S_{W_3})=  \Cl(S_{W_2})=  \Cl(S_{W_1})$.

\subsubsection{Case $W=W_1$} 

%

We first focus on characterizing the semigroup generated by $W:=\{\exp(a X_1 +b X_2):a\in \R, b\geq 0\}$. 

Using the group law \eqref{gp_law} (or alternatively  integrating the relevant ODE of the vector fields \eqref{LIVF_F23})  we have 
\begin{equation}\label{un_flusso}
\exp(t(a X_1+X_2)) =\left(at, t, -\frac{1}{2}at^2, \frac{1}{6}a^2 t^3, \frac{1}{3}at^3\right).\end{equation}


Using this formula, we shall compute the coordinates of the point $$\exp(t_1(a_1 X_1+X_2)) \exp(t_2(a_2 X_1+X_2)),$$ where $a_1, a_2 \in \R$ and $t_1, t_2 >0$; in other words, we are performing a zig-zag, starting at $0$ and flowing first for time $t_1$ along $a_1 X_1 + X_2$ and then (from the point $\exp(t_1(a_1 X_1+X_2))$ just reached) flowing for time $t_2$ along $a_2 X_1 + X_2$. Clearly, every point reached by any number of such zig-zags is in the semigroup.

The computation of $\left(a_1t_1, t_1, -\frac{1}{2}a_1t_1^2, \frac{1}{6}a_1^2 t_1^3, \frac{1}{3}a_1t_1^3\right)\cdot \left(a_2t_2, t_2, -\frac{1}{2}a_2t_2^2, \frac{1}{6}a_2^2 t_2^3, \frac{1}{3}a_2t_2^3\right)$ gives

$$\left(a_1 t_1 + a_2 t_2, t_1+t_2, -\frac{a_1 t_1^2}{2}-a_1 t_1 t_2-\frac{a_2 t_2^2}{2}, \frac{a_1^2 t_1^3}{6}+\frac{a_1^2 t_1^2 t_2}{2} + \frac{a_1 a_2 t_1 t_2^2}{2} + \frac{a_2^2 t_2^3}{6}, \right.$$ $$\left.\frac{a_1 t_1^3}{3} + a_1 t_1^2 t_2+ \frac{a_1 t_1 t_2^2}{2} + \frac{a_2 t_1 t_2^2}{2}+ \frac{a_2 t_2^3}{3}\right).$$

Since $X_1 = \partial_1$, we know that the semigroup is invariant in $x_1$ and we can therefore neglect the first coordinate obtained (since we are interested in characterizing the points in $\R^5$ that can be reached by zig-zags, given a point in the image we can change the first coordinate fixing the others by flowing along $X_1$); moreover, using the fact that the semigroup is invariant under subRiemannian dilations, and noticing that the only point in the semigroup with 0 second coordinate are those other form $(x_1,0,\ldots,0)$, we can fix the second coordinate to be $1$. (This forces the choice $t_2=1-t_1$; to reach such a point, we need to correctly scale the zig-zag.) We therefore have a map from $\R^2 \times [0,1]$ to $\R^5$, where the point $(a, b, c)$ in the domain is mapped to the point in $\R^5$ that is reached by the above zig-zag (we can fix the first coordinate to be $0$ by the additional flow along $X_1$) choosing $a_1=a$, $a_2 =b$, $t_1=c, t_2=1-c$ (by the choices made, this point will lie in the $3$-dimensional plane $\{x_1=0, x_2=1\}$, and we identify this with the target $\R^3$).

After substituting the variables and simplifying, the map $G:\R^2 \times [0,1] \to \R^3$ thus obtained reads

$G(a,b,c)=\hspace{6cm}$
$$ \left(\frac{1}{2}\left(a(c-2)c - b (1-c)^2\right), \frac{1}{6} \left( a^2(3-2c)c^2 + 3abc(c-1)^2 - b^2(c-1)^3\right), \right.$$ $$\hspace{9cm}\left. \frac{1}{6}\left(b(c^3-3c+2)-ac(c^2-3)\right)\right).$$
The Jacobian matrix of $G$ is
$$\hspace{0cm} \left[\begin{array}{ccc}
   \frac{(c-2)c}{2}  & -\frac{(c-1)^2}{2} & \frac{a(c-2)-2b(c-1)+ac}{2}\\
   \frac{3bc(c-1)^2+2a(3-2c)c^2}{6} & \frac{3a(c-1)^2c - 2b(c-1)^3}{6} & \frac{-2c^2 a^2 + 2(3-2c)ca^2 + 3b(c-1)^2a+6abc(c-1)}{6} \\
  -\frac{c(c^2-3)}{6} & \frac{c^3 -3c+2}{6} & \frac{-2ac^2 - a(c^2-3)+b(3c^2-3)}{6}
  \end{array} \right]
 $$
The Jacobian determinant of $G$ can be computed to be $\frac{1}{72}(c-1)^4 c^3 (a-b)^2$. This shows that the map $G$ is a diffeomorphism on $(\R^2 \times (0,1) ) \setminus \{(a,b,c):a=b\}$. 

Restricting to the set of critical points, we compute the image via $G$ of the surfaces $\{(a,b,0)\}$, $\{(a,b,1)\}$, $\{(a,a,c)\}$ and obtain respectively the curves in $\R^3$ expressed (parametrically) by
$$ (-\frac{b}{2}, \frac{b^2}{6}, \frac{b}{3})\,\,, \,\,(-\frac{a}{2}, \frac{a^2}{6}, \frac{a}{3})\,\,, \,\, (-\frac{a}{2}, \frac{a^2}{6}, \frac{a}{3}),$$
i.e., the curve $\{y=-xz\} \cap \{ 2x+3z=0\}$ in all three cases, using coordinates $(x,y,z)$ in the target $\R^3$.

For $(x,y,z)$ arbitrary, solving the system $G(a,b,c)=(x,y,z)$ (inverting $G$) yields the solution

$$\begin{array}{ccc}
   a=-\frac{2(2x^3 - 9xy-9yz)}{3(2x^2+6xz-y+6z^2)} & b=-\frac{3(4x^4+8x^3z-12x^2 y  - 36 xyz - 3y^2 - 18 yz^2)}{2(2x+3z)^3} &  c=-\frac{3(2x^2+6xz-y+6z^2)}{2x^2+6xz+3y},
  \end{array} 
 $$
 which is well-defined as long as the denominators do not vanish and as long as the third fraction is in the interval $[0,1]$ (since we restricted $c\in[0,1]$ in the definition of $G$). We already characterized the image of points with $c=0$ or $c=1$, so regarding $c$ we only need to understand when the third fraction is in $(0,1)$.
 
 Algebraic manipulation shows that 
 $$-\frac{3(2x^2+6xz-y+6z^2)}{2x^2+6xz+3y} = \frac{3\hat P}{3 \hat P + 2(2x+3z)^2},$$
 where $\hat P$ is the polynomial
 $$\hat P=\hat P(x,y,z):=y - 2x^2 - 6xz - 6z^2,$$
 therefore $c\in(0,1)$ if and only if $\hat P>0$. Note that $\hat P$ is also the denominator of the first fraction above, up to a multiplicative constant, therefore $\hat P>0$ guarantees that $a$ is well-defined. The only constraint left is that the denominator of $b$ be non-zero. 
 
 So far we have established that $\{(x,y,z):\hat P>0\} \setminus \{(x,y,z):2x+3z=0\}$ is in the image of $G$ restricted to $(\R^2 \times (0,1) ) \setminus \{(a,b,c):a=b\}$. 
 
 Consider the composition $Q\circ G$, where $Q(x,y,z)=2x+3z$. The function $Q\circ G$ vanishes at $(a,b,c)$ if and only if $G(a,b,c)\in \{2x+3z=0\}$. On the other hand, explicit computation gives $Q\circ G=-\frac{1}{2}(c-1)^2 c (a-b)$ and this function vanishes exactly on the set $\{a=b\} \cup \{c=0\} \cup \{c=1\}$, which is the set of critical points. We can therefore conclude that the image of $G$ restricted to $(\R^2 \times (0,1) ) \setminus \{(a,b,c):a=b\}$ is exactly the set $\{(x,y,z):\hat P>0\} \setminus \{(x,y,z):2x+3z=0\}$. 
 
 Putting all together, we have obtained that $G$ maps $\R^2 \times [0,1]$ to the set 
 $$\text{Im}(G)=\left(\{(x,y,z):\hat P>0\} \setminus \{(x,y,z):2x+3z=0\}\right) \cup \left(\{y=-xz\} \cap \{ 2x+3z=0\}\right).$$
 This is the upper-graph of a paraboloid ($\hat P>0$) from which we remove the plane $\{(x,y,z):2x+3z=0\}$ except for the curve $\{y=-xz\} \cap \{ 2x+3z=0\}$, which lies on the paraboloid. 
 
 Recalling the choices made of neglecting $x_1$ and fixing $x_2=1$, the above result says that if we embed $\text{Im}(G)\subset \R^3$ in $\{(0,1)\} \times \R^3\subset\R^5$ and extend it invariantly in $x_1$ and homogeneously (in the subRiemannian sense) in $x_2$ we obtain the set of points that can be reached using a zig-zag consisting of at most $3$ steps. We denote this set by $\tilde{S}$.
 Namely,
 $$\tilde{S} :=\{ (x_1    ,0, \lambda^2 x,\lambda^3 y,\lambda^3 z) : x_1\in \R, \lambda>0, (x,y,z)\in  \text{Im}(G)\}.$$

 \medskip
 
 We next claim that the semigroup $S$ generated by $\{\exp(a X_1 +b X_2):a\in \R, b\geq 0\}$ is the extension (invariantly in $x_1$ and homogeneously in the subRiemannian sense in $x_2$) of the set 
 $$ \{(x,y,z):\hat P>0\} \cup \left(\{y=-xz\} \cap \{ 2x+3z=0\}\right).$$
 To see this, we argue by first showing that the subvariety $R$ defined by the extension (invariantly in $x_1$ and homogeneously in the subRiemannian sense in $x_2$) of $ \{(x,y,z):\hat P>0\} \cap \{(x,y,z):2x+3z=0\}$, belongs to the semigroup. Pick any point $p\in R$ and consider $p S^{-1}$: if $p \notin S$ then we would have that $p S^{-1}$ cannot intersect $S$, in particular $p S^{-1}$ has to be contained in $R$. We know, however, that $S$, and a fortiori $p S^{-1}$, has non-empty interior (by the general theory or just by noticing that $\text{Im}(G)$ has non-empty interior) and so $p S^{-1}$ cannot be contained in $R$.
 
In the previous argument, one can  replace  $S$ with $\tilde{S}$, which also has non-empty interior. Consequently, we get the extra piece of information that any point in 
$$ \{(x,y,z):\hat P>0\} \cup \left(\{y=-xz\} \cap \{ 2x+3z=0\}\right)$$ can be reached using a zig-zag consisting of at most $6$ steps.
 
 The set 
 $$P(x):=x_2^3x_4 -  2x_2^2 x_3^2 - 6 x_2x_3 x_5 - 6 x_5^2>0, \,\,\, x_2\geq 0,$$
(this is the homogeneous polynomial of degree 6 obtained by extending the above $\hat P$) has constant normal equal to $X_2$. One way to check this, is by computing the derivative of $P(x)$ in the direction $X_{2}=\partial_{2} -x_{1}\partial_{3} + \frac{x_{1}^{2}}{2}\partial_{4} + x_{1}x_{2}\partial_{5}$, which gives
$$A=\frac{1}{2} x_2^3 x_1^2 + (-10x_2^2 x_3 -18 x_2 x_5)x_1 + (3x_2^2 x_4 -4 x_2 x_3^2 -6x_3 x_5);$$
viewing this as a polynomial of degree $2$ in $x_1$, the discriminant is $-6(x_4 x_2^5-2x_3^2 x_2^4 - 6 x_3 x_5 x_2^3 - 6 x_5^2 x_2^2)=-6x_2^2 P(x)$, which vanishes at $P=0$ and is negative on $P>0$. Therefore, since for $x_2\geq 0$ the polynomial $A$ goes to $+\infty$ for $x_1\to \pm \infty$, we get that $A$ is always non-negative when $P\geq 0$, in other words the $X_2$-monotonicity is satisfied. Also note that when $P>0$ and $x_2>0$ we have $A>0$, i.e., the derivative is strictly positive, which means that the flow along $X_2$ increases the value of $P$, in particular if we start from an interior point of $P\geq 0, x_2\geq 0$, then the flow cannot converge to any point on $P=0$.
This implies that the points of $S$ that lie on the boundary of the paraboloid cannot be reached from points in the interior.

Note that the critical curve $\{y=-xz\} \cap \{ 2x+3z=0\}$ extended invariantly in $x_1$ and homogeneously in $x_2$ is a subvariety $Z$ of (topological) dimension $3$ and coincides with $\exp(W) \exp(\R X_1)$: indeed, for the critical points $c=0$ or $c=1$ then we are choosing one of $t_1, t_2$ to be $0$, i.e., we are only flowing once (for time $1$) along $aX_1+X_2$ (for some choice of $a$). For the critical points $a=b$ we are flowing twice along the same vector $aX_1+X_2$, equivalently we are flowing along $aX_1+X_2$ for time $1$. Recalling that with $G$ we had restricted $x_2=1$ and neglected $x_1$, the previous observations say that $Z$ corresponds to the points reached by a flow along a vector in $W$ followed by a flow along $X_1$, i.e., $Z=\exp(W) \exp(\R X_1)$ as claimed.

What is left to prove, in order to complete the characterization of $S$, is that no other point with $P=0$, other than the points on $Z$, is in $S$. 
We established that if we start from an interior point then the flow will not converge to a boundary point, so the only way to reach a point on $P=0$ is to flow from a point that lies on $P=0$; this means, in view of the characterization of $\text{Im}(G)$ above, that we must start from a point in $Z$. However, $\exp(W) \exp(\R X_1)$ is invariant for the flow, so the argument is complete.

\medskip

We summarize our conclusions with the following statement.
\begin{proposition}\label{S:exp2}
In the free group $\mathbb{F}_{23}$ equipped, as above, with exponential coordinates of the second kind,
the semigroup $S$ generated by $W:=\{\exp(a X_1 +b X_2):a\in \R, b\geq 0\}$ is $$S=\{
x\in \R^5 
: P(x):=x_2^3x_4 -  2x_2^2 x_3^2 - 6 x_2x_3 x_5 - 6 x_5^2>0, \,\,\, x_2> 0\}
\cup \exp(W) \exp(\R X_1).$$
The set $\exp(W) \exp(\R X_1)$ has the coordinate expression $$\{P=0\} \cap \{x_2\geq 0\} \cap \{x_2^2 x_4=-x_3 x_5 \} \cap \{ 2x_2 x_3+3 x_5=0\}.$$
Every element in $S$ can be written as the product of 6 elements in $W$, i.e., $W^6=S$.
The set $W^3$ is a set with interior that is different than $S$, but its closure is $\overline{S}$.
\end{proposition}

\subsubsection{Case $W=W_2$}  

Regarding the case $W:=\{\exp(a X_1 +b X_2):a\in \R, b> 0\}$, we clearly obtain a semigroup $S_{W}$  that is contained in the previous semigroup $S_{W_1}$. Nonetheless, the set $  \exp(W) \exp(\R X_1)$ is in such semigroup, but it is not in the interior, since it is not in the interior of $S_{W_1}$.
We therefore understand that, also in this case, the set $S_{W}$ is not open in the Carnot group, even if $W$ is   open in $V_1$. 
We expect that in fact $S_{W}=S_{W_1}\setminus \exp(\R X_1)$. 

\subsubsection{Case $W=W_3$} \label{sec:caso3}

Since we are in a free group, instead of studying the semigroup generated by $W_3$ from \eqref{Wi} we may study the one generated by  
$$
W := \{a X_1 :a\in \R \} \cup \{  b (X_1+ X_2) : b> 0\}
.$$
We are going to show that the point $\exp(X_2)$ is in the closure of $S_W$,
 but not in any $(\exp(W))^k$, with $k\in \N$. Consequently, we have strict inclusions
$$(\exp(W))^k\subsetneq S_W,\quad 
\forall k\in \N.$$ 
Since the sets $(\exp(W))^k$ are increasing, we can clearly restrict to even values of $k$. Then, for $k\in \N$, an arbitrary point in $(\exp(W))^{2k}$ has the form
\begin{equation}\label{un_punto}q:= \exp(a_1 X_1) \exp(b_1 (X_1+ X_2) ) \ldots \exp(a_k X_1) \exp(b_k (X_1+ X_2) )   ,
\end{equation}
with $a_j \in \R$ and $b_j> 0$, for $j=1,\ldots, k$. 

We now check that this point $q$ differs from $\exp(X_2)$ in the coordinates.
Using the group law \eqref{gp_law} and the calculation in \eqref{un_flusso}, 
we get that when a point $p=(p_1,\ldots, p_5)$ is right multiplied by 
$\exp(t( X_1+X_2))$ then the first and fourth components of the point obtained are respectively
$$(p\exp(t( X_1+X_2)))_1=p_1+t,$$
$$(p\exp(t( X_1+X_2)))_4 = p_4 +\frac{(p_1+t)^3}{6}- \frac{p_1^3}{6}.$$
If $t>0$, we can bound
$$(p\exp(t( X_1+X_2)))_4 =
 p_4 +\frac{1}{2}t (p_1+\frac{t}{2})^2+ \frac{1}{24}t^3 
 \geq p_4  + \frac{1}{24}t^3 .$$
 Recall that instead $p\exp(aX_1)= (p_1+ a, p_2, \ldots, p_5)$.
We deduce that the point $q$ from \eqref{un_punto} satisfies
$$q_2= b_1+\ldots+b_k$$
$$q_4   \geq   \frac{1}{24}(b_1^3+\ldots+b_k^3) .$$
Now, if $q$ is close enough to $\exp(X_2)= (0, 1, 0, \ldots, 0)$, then there exists some $j$ such that 
$b_j>\frac{1}{2k}$, since the $b_i$'s are positive.
Consequently, $q_4   \geq   \frac{1}{24}(\frac{1}{2k})^3$. We found a bound away from 0 of $q_4$. Therefore there 
is no sequence of $q$'s accumulating to $\exp(X_2)$.

\subsection{Exponential coordinates (of first kind)}\label{sec:coords}
We next rephrase the previous results in the standard exponential coordinates. The purpose is to calculate the wedge of the semigroup $S$, which we calculated in Proposition~\ref{S:exp2}  using exponential coordinates of the second kind. 

Recall that the Baker-Campbell-Hausdorff formula for step-$3$ groups reads as
$$\exp X\exp Y= \exp\left(X + Y + \frac{1}{2}[X,Y] + \frac{1}{12}[X,[X,Y]] - \frac{1}{12}[Y,[X,Y]] \right).$$

Using this formula, we compute the products in exponential coordinates of first kind:
$$\exp(a_1 X_1+\ldots+a_5 X_5)\cdot \exp(b_1 X_1+\ldots+b_5 X_5)=\qquad\qquad\qquad\qquad\qquad\qquad\qquad\qquad $$
$$\qquad=
\exp\left( 
(a_1+b_1)  X_1+
(a_2+b_2) X_2+
(a_3+b_3+\dfrac{a_2b_1- a_1b_2}{2})  X_3 \right.\qquad\qquad$$ 
\begin{equation}\label{gp_law2}\qquad\qquad+
(a_4+b_4+ \dfrac{a_3b_1-a_1b_3}{2} + \dfrac{a_1(a_1b_2-a_2b_1)}{12} + \dfrac{b_1(a_2b_1-a_1b_2)}{12} )  X_4
\end{equation} $$
\left.\qquad\qquad\qquad\qquad\qquad
+(a_5+b_5 + \dfrac{a_3b_2-a_2b_3}{2} - \dfrac{a_2(a_2b_1-a_1b_2)}{12} + \dfrac{b_2(a_2b_1-a_1b_2)}{12}   )  X_5\right).
$$

Consequently, we relates the  exponential coordinates of first kind to the exponential coordinates of second kind:
If ${x}_1,\ldots, {x}_5$ are the coordinates of second kind and
$a_1,\ldots, a_5$ are the coordinates of first kind,
then the change of coordinates are the following:
$$
\exp({x}_5 X_5)\exp({x}_4 X_4)\cdot\ldots\cdot\exp({x}_1 X_1)=
\exp\left( a_1 X_1+a_2 X_2 + a_3 X_3 +a_4 X_4+a_5 X_5\right) $$
if and only if
\begin{equation*}
\left\{\begin{array}{ccl}
a_1 &=&{x}_1 \\
a_2 &=&{x}_2\\
a_3 &=&{x}_3+\frac{{x}_1{x}_2}{2}\\
a_4 &=&{x}_4+\frac{{x}_1{x}_3}{2}+\frac{{x}_1^2{x}_2}{12}\\
a_5 &=&{x}_5+\frac{{x}_2{x}_3}{2} - \frac{{x}_1{x}_2^2}{12},
\end{array}
\right.
\end{equation*}
or, equivalently,
\begin{equation*}
\left\{\begin{array}{ccl}
{x}_1    &=&a_1\\
{x}_2    &=&a_2\\
{x}_3    &=&a_3 -\frac{ a_1 a_2}{2}\\
{x}_4    &=&a_4 +\frac{1}{6} a_1^2 a_2 - \frac{1}{2} a_1 a_3\\
{x}_5    &=&a_5+\frac{1}{3} a_1 a_2^2 - \frac{1}{2} a_2 a_3.
\end{array}
\right.
\end{equation*}
In other words, the above two systems define the change of coordinates.
Notice that, using this formulas together with \eqref{gp_law2},
one obtains \eqref{gp_law}.


\subsection{The semigroup in exponential coordinates and its wedge}
In Proposition~\ref{S:exp2} we calculated the semigroup $S$ generated by $W:=\{\exp(a X_1 +b X_2):a\in \R, b\geq 0\}$ 
in exponential coordinates of the second kind. Using the above change of variables we claim  that   
in exponential coordinates (of first kind) the 
   semigroup is 
$$S=\{
a\in \R^5
: -\frac12 a_2^2 a_3^2 +  a_2^3 a_4 -  a_1 a_2^2 a_5 -
6 a_5^2  >0, \,\,\, a_2> 0\} \cup
\exp(W) \exp(\R X_1).$$
Indeed, one just has to verify that the polynomial $P(x):=x_2^3x_4 -  2x_2^2 x_3^2 - 6 x_2x_3 x_5 - 6 x_5^2$ from Proposition~\ref{S:exp2} in these other coordinates $(a_1, a_2, a_3, a_4, a_5)$ becomes 
 \begin{equation}\label{Polyexp1}
\tilde P(a) = -\frac12 a_2^2 a_3^2 +  a_2^3 a_4 -  a_1 a_2^2 a_5 -
6 a_5^2.
 \end{equation}
 Recall from Proposition~\ref{S:exp2} that 
$ \exp(W) \exp(\R X_1)\subseteq \{\tilde P=0\}.$

With these expression above, it is easy to calculate the wedge of the semigroup according to Definition~\ref{def:wedge}. In fact, we claim that 
the wedge is 
\begin{equation}\label{wedgeF23}
 \mathfrak w _{\bar S}:=\{ a_5=0, a_2= 0\}\cup \{ a_5=0, a_2> 0,   2 a_2 a_4 \geq a_3^2\} .
 \end{equation}
Indeed, 
first observe that $\log( \bar  S) = \{\tilde P \geq 0, a_2\geq 0\}$ and so
if $X\simeq (a_1, a_2, a_3, a_4, a_5)$ is such that  $\R_+ X \subseteq \log(  \bar S) $ then  
$$\tilde P(t a) = -\frac12 t^4 a_2^2 a_3^2 + t^4 a_2^3 a_4 -  t^4a_1 a_2^2 a_5 -
6 t^2a_5^2\geq 0 ,\qquad \forall t >0.$$
Hence,
$$0\leq \lim_{t\to 0^+ }\frac{1}{t^2} \tilde P(t a) = -
6  a_5^2,$$
which implies $a_5=0$.  
The polynomial in \eqref{Polyexp1} when $a_5=0$ becomes
$ -\frac12   a_2^2 a_3^2 +   a_2^3 a_4  \geq 0$.
If $a_2=0$ that we have no extra condition. Instead, when $a_2>0$, then we get
$2 a_2 a_4 \geq a_3^2$.
Hence, one inclusion in  \eqref{wedgeF23} is proved. Regarding the other inclusion, just observe that the right-hand side of  \eqref{wedgeF23} is a cone in $\log( \bar  S)$.
Hence \eqref{wedgeF23} is proved.

We shall now verify that $ \mathfrak w _{\bar S } \cap \Int(  S  ) \neq \emptyset $, and verify Conjecture~\ref{conj_intersect} in $\mathbb{F}_{23}$.
Indeed, in exponential coordinates this intersection is 
\begin{eqnarray*}
 \mathfrak w _{\bar S} \cap \Int(  S )
 &=&(\{ a_5=0, a_2= 0\}\cup \{ a_5=0, a_2> 0,   2 a_2 a_4 \geq a_3^2\} )
\cap(\{
\tilde P(a) >0 ,   a_2> 0\})\\
&=& \{ a_5=0, a_2> 0,   2 a_2 a_4 > a_3^2\}.
\end{eqnarray*}

\begin{remark}One pathology of the above semigroup $S$ is that $\exp(W)$ is contained in its boundary. Even worst, we claim that the set 
$\Ad_{\exp(\R X_1)}(\R X_1 + \R_+ X_2)$, which is in  $\mathfrak w _{\bar S_W}$ by Proposition~\ref{prop more elements in wedge} is not in the  intersection  $ \mathfrak w _{\bar S} \cap \Int(  S )$.
Indeed, for all $a,b
\in 
\R$ and $c>0$, we have 
\begin{eqnarray*}
\Ad_{\exp(a X_1)}(b X_1 + c X_2) &=& b X_1 + c X_2 + [ a X_1, b X_1 + c X_2] +\frac12 [ a X_1, [ a X_1 ,  b X_1 + c X_2]]\\
&=& b X_1 + c X_2 -ac  X_3 +\frac12 a^2 c  X_4\\
&=&( b , c , -ac  ,\frac12a^2 c  ,0).\\
\end{eqnarray*}
This point is not in the interior of $S$ since for this point we have 
$$\tilde P(a)  = -\frac12  a_2^2 a_3^2 +   a_2^3 a_4 
=
-\frac12   c^2 (-ac)^2 + c^3 (\frac12a^2 c) 
= 0.$$
Nonetheless, the point 
$X_2 + \Ad_{\exp( X_1)}(    X_2) = 2    X_2 -   X_3 +\frac12    X_4 = (0,2,-1,1/2,0)$ is the sum of two elements in 
$ \mathfrak w _{\bar S}$ and so it is in $ \mathfrak w _{\bar S}$, and, moreover, it is
 in the interior of $\log(S)$ since the value of $\tilde P$ is 
$-\frac12  2^2 (-1)^2 +   2^3 \frac12  = -2+4=2>0$.

More explicitly, one can check that the point 
$\exp(X_2 + \Ad_{\exp( X_1)}(    X_2)) $  is in $S$ since it can be written as
\begin{equation}\label{point_in_F23}
\exp( X_2 + \Ad_{\exp( X_1)}(    X_2) )  =
  \exp(    \tfrac{1}{2} X_2)  \exp(   X_1)  \exp( X_2)  \exp( -  X_1)  \exp( \tfrac{1}{2} X_2) . 
\end{equation}
Moreover,
the point $\exp(X_2 + \Ad_{\exp( X_1)}(    X_2)) $  is in the interior of $S$ since the map
$$(\eps_1, \ldots, \eps_5) \mapsto   \exp(    (\tfrac{1}{2}+\eps_1) X_2)  \exp(  (1+\eps_2) X_1)  \exp( (1+\eps_3)X_2)  \exp( (-1+\eps_4)  X_1)  \exp( (\tfrac{1}{2}+\eps_5)X_2) 
$$
is a submersion at $0$.
Indeed, a calculation gives that the images under the differential at 0 of such a map of the basis vectors
$\partial_{\eps_1}, \ldots, \partial_{\eps_5}$ are the left-pushed forward by $ \exp(X_2 + \Ad_{\exp( X_1)}(    X_2)) $ of
$$
X_2 +X_5,\,
X_1 -\tfrac{3}{2} X_3 +X_4-\tfrac{9}{8} X_5 ,\,
X_2   - X_3 +\tfrac{1}{2} X_4-\tfrac{1}{2} X_5  ,\,
X_1  -\tfrac{1}{2} X_3 -\tfrac{1}{8} X_5  ,\,
X_2
  .$$
These last 5 vectors are a basis of the Lie algebra. Hence there is a neighborhood of the point   \eqref{point_in_F23} that is inside $S$.
\end{remark}

%

\begin{remark}
\label{rem:evidence}

A similar calculation can be also done in the
free-nilpotent  group of rank 2 and step 4. Indeed, in the Carnot group $\mathbb{F}_{24}$ the point
$$\hspace{-7cm}\exp( 2 X_2 + \Ad_{\exp(-\tfrac{1}{2} X_1)}(    2X_2)  + \Ad_{\exp(\tfrac{1}{2} X_1)}(    X_2) ) $$ $$ \qquad=
  \exp(    X_2)  \exp(    -\tfrac{1}{2}   X_1)  \exp( X_2)  \exp(   X_1)  \exp( X_2)  \exp( -  X_1)  \exp( X_2)  \exp( \tfrac{1}{2}  X_1)  \exp(  X_2) 
$$
lies in $ \mathfrak w _{\bar S_W} \cap \Int(  S_W ) $, where $W$ is the horizontal half space with normal $X_2$. We omit the calculation, being similar to the one done above in the case of $\mathbb{F}_{23}$.
\end{remark}

%
%


\subsection{Some examples of constant normal sets} 
We provide in this section some examples of constant-normal sets in $\mathbb{F}_{23}$, with normal $X_2$. In doing so, we also point out some sufficient conditions for a set to satisfy the constant-normal condition.

\begin{example}\label{ex:Calphabeta}
For $\alpha, \beta\geq 0$, the set
$$  {\mathcal C}_{\alpha, \beta}:=\left\{\;x\in\R^5\;:\; x_2\geq 0, \;(\alpha\; x_3+\beta\; x_5)^2\leq 2\;x_2 \;x_4 \;(\alpha+\beta \;x_2)^2\;\right \}.$$
  is a constant-normal set with normal $X_2$.
  
  
  
   This two-parameter family of examples provides an interpolation between the two cones $\mathcal{C}_{0,1}=\{x: x_2\geq 0, x_5^2 \leq 2 x_2^3 x_4\}$, which  will be important in Sections~\ref{sec:controesempio1} and \ref{sec:controesempio2}, and $\mathcal{C}_{1,0}=\{x: x_2\geq 0, x_3^2 \leq x_2 x_4 \}$, which is the lift of the cone in the Engel group, as analysed in \cite[Example 3.31]{Bellettini-LeDonne}, recalling that
the Engel group is naturally a quotient of $\mathbb{F}_{23}$.
\end{example}

To check the condition, we compute the derivative of the polynomial $P(x)=2\;x_2 \;x_4 \;(\alpha+\beta \;x_2)^2-(\alpha\; x_3+\beta\; x_5)^2$ in the direction $X_{2}=\partial_{2} -x_{1}\partial_{3} + \frac{x_{1}^{2}}{2}\partial_{4} + x_{1}x_{2}\partial_{5}$, which gives
$$d(x)=2 x_4(\alpha+\beta x_2)^2 + 4 \beta x_2 x_4(\alpha + \beta x_2)+2 \alpha x_1 (\alpha x_3 + \beta x_5) + x_1^2 x_2 (\alpha+\beta x_2)^2 - 2 \beta x_1 x_2 (\alpha x_3 + \beta x_5).$$
We will check that $d(x)$ is non-negative on the set $P(x)\geq 0$. Viewing the expression for $d(x)$ as a polynomial of degree $2$ in $x_1$, the discriminant is given by
$$D=(\alpha x_3 + \beta x_5)^2 (\alpha-\beta x_2)^2 - 2 (\alpha+\beta x_2)^3 x_2 x_4 (\alpha + 3 \beta x_2).$$
We claim that $D \leq 0$ on the set $\{P\geq 0\}$: this concludes the proof because the coefficient of $x_1^2$ in $d(x)$ is non-negative ($x_2 \geq 0$ and the only cases in which the coefficient may vanish correspond to $E$ being equivalent to the empty set) so that the $X_2$-flux starting from a point in $\mathcal{C}_{\alpha,\beta}$ must always remain (for positive times) in $\mathcal{C}_{\alpha,\beta}$. To check the claim, note that $P\geq 0$ implies 
$$D \leq 2\;x_2 \;x_4 \;(\alpha+\beta \;x_2)^2 (\alpha-\beta x_2)^2 - 2 (\alpha+\beta x_2)^3 x_2 x_4 (\alpha + 3 \beta x_2)$$
so we only need to check that $(\alpha-\beta x_2)^2 \leq (\alpha+\beta x_2) (\alpha + 3 \beta x_2).$ This is immediate by expanding the terms and using $\alpha, \beta, x_2 \geq 0$. 
 
 \medskip
 The remainder of this section provides further examples of regular sets with constant normal, pointing out some methods to construct them. It is not necessary for the sequel of the paper: the reader interested in the pathological examples may skip to Section~\ref{sec_ex}.
 
 
\begin{remark}
 \label{remark:upper-graph}
All constant normal sets in rank $2$ are represented, in exponential coordinates of II type, by upper-graphs of functions that are independent of the $x_1$-variable. More precisely, assuming without loss of generality that the normal is $X_2$ and choosing a representative $E$ that is precisely $X_2$-monotone and open (this exists by Theorem~\ref{main:thm2}), then 
$$E= \{ x_2 > G(x_3, \ldots, x_n) \},$$ 
with $G$ upper-semi-continuous.

To prove this, notice that $X_2|_{\{x_1=0\}}= \partial_2$ and thus $(0, x_2, x_3, \dots, x_n)\exp(tX_2)=(0, x_2+t, x_3, \dots, x_n)$; the precise $X_2$-monotonicity implies that the function $G:\R^{n-2}\to \overline{\R}$ defined by $G(x_3, \dots, x_n) := \inf\{x_2: (0, x_2, x_3, \dots, x_n) \in E\}$ (with the convention $\inf\{\emptyset\}=+\infty$) has the property that $E\cap \{x_1=0\}$ is given by $\{0\} \times \{x\in \R^{n-1}: x_2 >G(x_3, \dots, x_n)\}$. The $X_1$-invariance (notice that $X_1= \partial_1$) gives that $E= \{ x_2 > G(x_3, \ldots, x_n) \}$. Since $E$ is open, $G$ is upper-semi-continuous.
\end{remark}

\begin{example}\label{example da PDI}
We give here a sufficient condition on a function $G:\R^3\to\R \cup \{ +\infty\}$, to ensure that the set $E \subset \mathbb{F}_{23}$ defined, in exponential coordinates of II type, by 

$$E= \{ x_2 \geq G(x_3, \ldots, x_n) \}$$
is a precisely constant-normal set with normal $X_2$. The sufficient condition entails that the set $\{G = +\infty\}$ is of the form $\{x_4 \leq a\}$ for some $a\in \overline{\R}$, $G$ is of class $C^1$ on $\{G \neq +\infty\}$ and $G$ satisfies, on the same set, the partial differential inequality
$$(\partial_3 G - G\; \partial_5 G)^2+2 \;\partial_4 G\leq 0.$$

The example $\mathcal{C}_{0,1}= \{x: x_2 \geq 0, x_5^2 \leq 2 x_2^3 x_4\}$ given above in Example~\ref{ex:Calphabeta} is of this type, choosing $G(x_3, x_4, x_5)= \sqrt[3]{\frac{x_5^2}{2 x_4}}$ for $x_4 >0$ and $G(x_3, x_4, x_5)=+\infty$ for $x_4\leq 0$. We will check this fact in Example~\ref{example B} below.
\end{example}

As before, we compute (on $\{G\neq +\infty\}$) the derivative of $P(x)= x_2 -G(x_3, x_4, x_5)$ in the direction $X_{2}=\partial_{2} -x_{1}\partial_{3} + \frac{x_{1}^{2}}{2}\partial_{4} + x_{1}x_{2}\partial_{5}$, which gives
$$d(x)=1+x_1 \partial_3 G - \frac{x_1^2}{2} \partial_4G -x_1 x_2 \partial_5 G.$$
Note that $\partial_4 G \leq 0$ by assumption, so the coefficient of $x_1^2$ is non-negative. If it vanishes, i.e.~if $\partial_4 G = 0$ at $(x_3,x_4,x_5)$ then the PDI gives $\partial_3 G = G \partial_5 G$, therefore in this case $d=1$ on $\{(x_1, G(x_3,x_4,x_5), x_3,x_4,x_5): \partial_4 G =0\}$ for any choice of $x_1$. We now consider the case $\partial_4 G < 0$ and view $d(x)$ as a polynomial of degree $2$ in $x_1$, with positive coefficient for the quadratic term. Its discriminant is given by
$$D=(\partial_3 G -x_2 \; \partial_5 G)^2 +2 \;\partial_4 G$$
and the PDI assumption gives $D\leq 0$ when $x_2= G$. This implies that, if $D<0$ at $(x_3,x_4,x_5)$ then $d >0$ everywhere on the line $(x_1, G(x_3,x_4,x_5), x_3,x_4,x_5)$; in the case $D = 0$ at $(x_3,x_4,x_5)$, then $d >0$ everywhere on the line $(x_1, G(x_3,x_4,x_5), x_3,x_4,x_5)$ except for the point corresponding to $x_1=\frac{\partial_3 G -G \; \partial_5 G}{\partial_4 G} = \pm \frac{\sqrt{2}}{\sqrt{-\partial_4 G}}$, at which $d=0$. The condition $d>0$ at $(x_1, G(x_3,x_4,x_5), x_3,x_4,x_5)$ implies that the $X_2$-flow starting at $(x_1,  G(x_3, x_4, x_5), x_3, x_4, x_5)$ enters the upper-graph: by continuity of $G$, the flow stays in the closed upper-graph in a neighbourhood of $(x_1,  G(x_3, x_4, x_5), x_3, x_4, x_5)$ also at the point where $d=0$. 

In conclusion, at all points $(x_1, G(x_3,x_4,x_5), x_3,x_4,x_5)$ where $G<\infty$ we have ensured that the $X_2$-flow starting at that point is contained in $\{x_2\geq G\}$ for an interval of time $[0, \eps)$ with $\eps>0$ (at this stage, $\eps$ may depend on the point). We now consider any $X_2$-flow starting at some point on $(x_1, G(x_3,x_4,x_5), x_3,x_4,x_5)$ (with $G<\infty$) and denote by $t_0>0$ the infimum of times for which the flow is outside $\{x_2\geq G\}$. The aim is to ensure that $t_0=+\infty$. If that were not the case, then the $X_2$-flow at $t_0$ would yield a point that is either on $(x_1, G(x_3,x_4,x_5), x_3,x_4,x_5)$ or in the domain where $G=\infty$. The first alternative is impossible by the previous discussion, since we ensured that the $X_2$-flow from any point $(x_1, G(x_3,x_4,x_5), x_3,x_4,x_5)$ stays in $\{x_2\geq G\}$ for some time. The second alternative is also impossible, by the assumption on $\{G=\infty\}$: the coordinate expression of $X_2$ shows that the $x_4$-coordinate is non-decreasing along the $X_2$-flow and we have started the flow from a point with $x_4 \geq a$. This concludes the proof of the precise $X_2$-monotonicity of $\{x_2 \geq G\}$.

\begin{remark}
 \label{remark:PDI}
In \cite{Bellettini-LeDonne} we showed that $E$ is a precisely constant-normal set in the Engel group \textit{if and only if}, in exponential coordinates of the second kind, $E$ is the upper-graph of a $BV$-function $G$ that satisfies a partial differential inequality, in which the partial derivatives (that are Radon measures) appear in duality with non-negative test functions (in the case in which $G$ is $C^1$, the PDI in \cite{Bellettini-LeDonne} bears similarities to the one in example~\ref{example da PDI}). As pointed out in Remark~\ref{remark:upper-graph}, it remains true in arbitrary Carnot groups that (in exponential coordinates of the second kind) any precisely constant-normal set $E$ can be written as $\{x: x_2> G\}$ for some function $G$. Therefore one could investigate the possibility of a characterization via PDI of constant normal sets in $\mathbb{F}_{23}$, in analogy with \cite[Theorem~3.17]{Bellettini-LeDonne}. We do not pursue this direction, and provide here a brief discussion of the issues involved.

It seems reasonable to expect the following: if $G:\R^3\to \R$ is $BV_{\rm{loc}}$ and satisfies a distributional analogue of the PDI in example~\ref{example da PDI} then (adapting the arguments in \cite[Theorem 3.17]{Bellettini-LeDonne}) the upper-graph of $G$ is a constant-normal set. However, this cannot provide a \textit{characterization} of constant-normal sets in $\mathbb{F}_{23}$. In fact, we will show in the Section~\ref{sec:controesempio1} that, in sharp contrast with the case of the Engel group, in $\mathbb{F}_{23}$ the function $G$ may fail to be in $BV_{\rm{loc}}$. If we wanted to  characterize the property of having constant normal in $\mathbb{F}_{23}$ by a PDI for $G$, then we would have to allow $G\notin BV_{\rm{loc}}$ in a distributional generalization of the PDI in example~\ref{example da PDI}. Understanding if this is possible requires a careful analysis, since the multiplication of $G$ by $\partial_5 G$ would have to be well-defined (if $\partial_5 G$ is an arbitrary distribution then the general theory only allows multiplication by a smooth function).
\end{remark}

\begin{example}\label{example B}
If $F:\R^2\to (0, +\infty]$, with coordinates $(x_4, x_5)$ on $\R^2$, is such that $F = +\infty$ on a set of the form $\{x_4 <a\}$ for $a\in [-\infty, \infty]$ and $F$ is $C^1$ on $\{x_4>a\}$ and satisfies on this set the PDI
$$(\partial_5 F)^2 + 6\; \partial_4 F\leq 0,$$
then
$$E=\{x\in\R^5\;:\;x_2>\sqrt[3]{F(x_4,x_5)}\}$$
 is a precisely constant-normal set with normal $X_2$ (with the convention $\sqrt[3]{+\infty}=+\infty$). This follows from Example~\ref{example da PDI}; in fact, setting $G=\sqrt[3]{F(x_4,x_5)}$, the PDI for $G$ from Example~\ref{example da PDI} is equivalent to $\frac{1}{9 F^{2/3}} \left((\partial_5 F)^2 + 6\; \partial_4 F\right)\leq 0$. 
 
 Choosing $a=0$ and $F(x_4, x_5)=\frac{x_5^2}{2 x_4}$ (note that $F>0$ on $\{x_4>0\}$), we have $(\partial_4 F)^2 + 6 \partial_4 F = (\frac{x_5}{x_4})^2 - 3(\frac{x_5}{x_4})^2 \leq 0$. We thus recover the example $\mathcal C_{0,1}$ given above in Example~\ref{ex:Calphabeta}.
\end{example} 

\begin{example}\label{examples} 
Constant-normal sets of the type in example~\ref{example B} are abundant. In fact, let $g:\R\to\R$ be a non-negative and non-increasing function and, for $C \in (0, \infty)$, let $f:\R\to \R$ be a non-negative and non-decreasing Lipschitz function with Lipschitz constant $\leq \frac{6}{C^2}$. Then $$E:=\left\{x\in\R^5 : x_2 >\sqrt[3]{ f(Cx_5 - x_4) + g(x_4)}\right\}$$ is a constant-normal set with normal $X_2$. 

If $f, g$ are $C^1$, then we are in the setting of Example~\ref{example B}. Set $F(x_4,x_5)=f(Cx_5 - x_4) + g(x_4)$. The validity of the PDI from example~\ref{example B} is easy to check: indeed, we get $(\partial_5 F)^2 + 6\; \partial_4 F = C^2 (f')^2  - 6 f'+ 6 g' = f' (C f' - 6) + 6 g' \leq  f' (C^2 f' - 6) \leq 0$, where we used $g'\leq 0$ and $0\leq f' \leq \frac{6}{C^2}$.

As mentioned in Remark~\ref{remark:PDI}, one could write an analogous PDI for $BV$ functions, which would allow $F(x_4,x_5)=f(Cx_5 - x_4) + g(x_4)$. For this specific example, however, it suffices to observe that smoothing $f$ and $g$ by convolution we preserve the sign and monotonicity properties, as well as the Lipschitz constant for $f$. We then obtain (smooth) constant-normal sets $E_n$ that converge to $E$ (in $L^1$) and the cone property passes to $E$. 

\end{example}

\section{Examples with pathologies} 
\label{sec_ex}
In this section, first we construct a
     constant normal set that does not have locally finite perimeter with respect to any   Riemannian metric, see Theorem~\ref{thm:no_BV}.
     Second,  we  construct an example with non-unique   blowup at some point, showing that it has     different upper and lower sub-Riemannian density at the point, see Theorem~\ref{prop:non_unique_tangents}. 
\subsection{Lack of the Caccioppoli set property in the Euclidean sense}  \label{sec:controesempio1}
This section is devoted to the construction of a set with constant normal that does not have locally finite Euclidean perimeter. The result is in Theorem~\ref{thm:no_BV}. Before proving it, we construct some auxiliary  sets.

We consider a Cantor set $K \subset [0,1]$ with $\mathcal{H}^1(K)>0$, constructed as follows. Let $a_0, a_1,\ldots$ be a sequence of positive real numbers such that
$$\sum_{j=0}^\infty a_j <1$$
and
$$\lim_{n\to \infty} 2^n a_n^2 = +\infty.$$
(For example, one can take a tail of the sequence $1/n^2$.)
We construct $K$ by induction.
We start with the set $[0,1]$ and remove the open segment of length $a_0$ and center in $1/2$.
We are left with two closed segments, say $I_{0,1}$ and $I_{0,2}$.
At step $n$ we have $2^n$ closed segments $I_{n-1,1},\ldots,I_{n-1,2^n}$.
From each $I_{n-1,k}$ we remove the open segment $J_{n,k}$
centered at the center of $I_{n-1,k}$ and
 of length $$| J_{n,k}|= a_n/2^n.$$
 Hence, we are left with $2^{n+1}$ segments    $I_{n,1},\ldots,I_{n,2^{n+1}}$ of length
 $$| I_{n,k}|= \dfrac{1- \sum_{j=0}^n a_j}{2^{n+1}}.$$
Let $K$ be the set of points $a\in [0,1]$ such that, for all $n\in N$,  there exist  $k$ for which $a\in I_{n,k}$. It is straightforward that $K$ is a Cantor set with positive measure.

\medskip

For $\mu>0$, we consider the set
$$W_{\mu}:=\{(x,y)\in \R^2 : (x-t)^2\leq \mu y^3, \text{ for some } t\in K\}.$$

\begin{lemma}
\label{Cantor}
Let $\mu>0$, and let $K$  and $W_{\mu}$ be as above. There exists a (H\"older) function $f_{\mu}:\R \to [0,\infty)$ such that 
\begin{itemize}
 \item[$\bullet$] $W_{\mu} = \{(x,y)\in \R^2 :y\geq f_{\mu}(x)\}$;
 \item[$\bullet$] $f_{\mu}$ is not $BV_{\rm{loc}}(\R)$.
\end{itemize}
\end{lemma} 

\begin{proof}
 
Define $$f_{\mu} = \inf_{t\in K} g_t(x), \text{ where } g_t(x)=\frac{1}{\mu^{1/3}}(x-t)^{2/3}.$$
We will drop the subscript $\mu$ within the proof, so we will write $W=W_\mu$ and$f=f_\mu$. Clearly $f\geq 0$. For $x\in \R$, if $y<f(x)$ then $y < g_t(x)$ for all $t\in K$, i.e., $(x-t)^2> \mu y^3$ for all $t\in K$, in other words $(x,y) \notin W$. If $y>f(x)$ then $y > g_t(x)$ for some
$t\in K$, i.e., $(x-t)^2< \mu y^3$ for some $t\in K$, in other words $(x,y) \in W$. The function $f$ is H\"older continuous since for every $t\in K$ we have 
$-\frac{1}{\mu^{2/3}}|x-y|^{2/3} \leq g_t(x) - g_t(y) \leq\frac{1}{\mu^{2/3}}|x-y|^{2/3}$. By continuity of $f$ we have $W= \{(x,y)\in \R^2 :y\geq f(x)\}$ and $\partial W= \text{graph}(f)$.

We will next show that $f$ fails to be differentiable at every point in $K$ (with $K$ contained in the domain of $f$). This will guarantee the second conclusion of the lemma in view of 
the fact that a $BV$ function of $1$ variable is the sum of two monotone functions and is therefore (classically) differentiable almost everywhere (we chose $K$ of positive measure).

Take $a\in K$. Let $I_n:=  I_{n,k(n)}$ such that $a\in I_n$, for all $n\in \N$.
Let $J_{n+1}$ be the interval removed from $I_n$ in the construction of $K$.
Let $q_n$ be the midpoint of $J_{n+1}$.
Let $p_n$ be the point in $\partial I_{n+1}$ that is between $a$ and $q_n$. 

Notice that $W \subset\{(x,y)\in\R^2:y\geq 0\} $
and $(a,0)\in \partial W$. So, if $f$ were differentiable at $a$ (recall that $\text{graph}(f)=\partial W$) then the derivative would be $0$ because $f$ vanishes on $K$ and no point in $K$ is isolated.
However, we shall show that, for a suitable $q_n'>0$, 
$(q_n,q_n')\in \partial W$
and the slope between $(a,0)$ and $(q_n,q_n')$ tends to $\infty$.
Indeed, , for  $q_n':= \mu^{-1/3} (| J_{n+1}| /2)^{2/3}$, 
we have the following properties.
The point $(q_n,q_n')$ belongs to $ W$, since, for $t=p_n$, 
$$(q_n-t)^2=(q_n-p_n)^2= \left( \dfrac{|J_{n+1}|}{2}\right)^2=
 \mu \left(  \mu^{-1/3} \left(\dfrac{|J_{n+1}|}{2}\right)^{2/3}\right)^3=\mu\, (q_n')^3 .$$
 For $t\in K\setminus\{p_n\}$,
 we have
 $|q_n-t|\geq |q_n-p_n|$ and so $(q_n-t)^2\geq\mu\, (q_n')^3 .$
Hence, $(q_n,q_n')\in \partial W$.

Regarding the slope, we can bound, for a suitable constant $k$ (depending only on $\mu$),
\begin{eqnarray*}
\dfrac{q'_n}{|a-q_n|}&\geq&\dfrac{  \mu^{-1/3} (| J_{n+1}| /2)^{2/3}}{   |I_n|}\\
&\geq& {  \mu^{-1/3} (a_{n+1}/2^{n+2})^{2/3}}{   \,\dfrac{2^{n+1}}{1- \sum_{j=0}^n a_j}}\\
&\geq&k\;2^{n/3}(a_{n+1})^{2/3}   , 
\end{eqnarray*}
which goes to $\infty$, for how the sequence $a_n$ was chosen. 
\end{proof}

%
%
%
  
We next exhibit a set in $\mathbb{F}_{23}$ that is $X_2$-monotone however it is not a set of locally finite perimeter in the Euclidean sense. We work in exponential coordinates of the second kind.  
  
\begin{theorem}\label{thm:no_BV}
Consider the set 
$$ \mathcal C:=\left\{\;x\in\R^5\;:\; x_2\geq 0, \;x_4\geq 0, \;(x_5)^2\leq 2\;(x_2)^{3} \;x_4\;\right \}.$$
Let $K$ be the Cantor set constructed before Lemma~\ref{Cantor} and define the set $$E:= \bigcup_{t\in K}L_{\exp(t X_5)}(\mathcal C).$$ Then $E$ is a set with constant normal, however $E$ is not a set with locally finite Euclidean perimeter.
\end{theorem}

\proof
The set $\mathcal C$ is the set $\mathcal C_{0,1}$ from Example~\ref{ex:Calphabeta}, which  we proved has precisely constant normal with normal $X_2$. 
Thanks to item (b) of Proposition~\ref{Remark_unioni}, we have that $E$ has precisely  constant normal with normal $X_2$.  It follows that $E$ is the upper-graph of a function $G:\R^4 \to \R$: to see this, note that $E \cap \{x_1=0\}$ is the upper-graph of a function because, for $x_1=0$, the vector field $X_2$ is just $\partial_2$, therefore the $X_2$ precise monotonicity gives the conclusion for $E \cap \{x_1=0\}$; the $X_1$-invariance gives the conclusion for $E$ (note that $G$ is independent of $x_1$).


To prove the theorem, note that, since $X_5$ is in the kernel of the Lie algebra, the left translation $L_{\exp(t X_5)}$ is an Euclidean translation. Namely, for all $t\in \R$ and $h\in \R^5$,
\begin{eqnarray*}
L_{\exp(t X_5)}(h)&=&\exp(tX_5) h\\
&=&\exp(tX_5)+h\\
&=&te_5+h.
\end{eqnarray*}
Let
$p_1, p_3\in \R$, $p_4>0$.
Define the $2$-dimensional plane
$$\Pi:=\{x\in \R^5: x_1=p_1, x_3=p_3, x_4=p_4\}.$$
Consider the set
\begin{eqnarray*}
E\cap\Pi &=& \bigcup_{t\in K}(\mathcal C+te_5)\cap \Pi\\
&=& \bigcup_{t\in K} 
\{x\in \R^5: x_1=p_1, x_2\leq0, x_3=p_3, x_4=p_4, (x_5-t)^2\leq 2\;(x_2)^{3} \;x_4  \}.
\end{eqnarray*}
Such a set is isometric to the set (using the notation introduced before Lemma~\ref{Cantor})
$$W_{2p_4}:=\{(a,b)\in \R^2 : a\geq 0, (b-t)^2\leq 2p_4 a^3, \text{ for some } t\in K\}.$$
If $E$ were a Caccioppoli set in the Euclidean sense, then the function $G:\R^4 \to \R$ such that $E = \{x_2 \geq G(x_1, x_3, x_4, x_5)\}$ would be $BV_{\rm{loc}}(\R^4)$, see \cite[Theorem 1, p.371]{Giaquinta-Modica-Soucek}. Then $\mathcal{L}^3$-almost every $1$-dimensional slice in the direction of $x_5$ would be a $BV_{\rm{loc}}$-function of $1$-variable by \cite[Theorems 5.21, 5.22]{Evans_Gariepy_revised} . 

By the characterization of $E\cap\Pi$ above, however, we see that the $1$-dimensional slice at $(p_1, p_3, p_4)$ is the function $f_{2p_4}$, with notation as in Lemma~\ref{Cantor}, and therefore (by the same lemma) it is not $BV_{\rm{loc}}(\R)$. This is the case for any choice of $(p_1, p_3, p_4)$ with $p_1, p_3\in \R$, $p_4>0$, therefore $G$ is not $BV_{\rm{loc}}(\R^4)$ and $E$ is not a Caccioppoli set.
%
\qed

\subsection{Non-uniqueness of subRiemannian tangents}
 \label{sec:controesempio2}
In this section we consider a variant of the example $E$ given in Theorem~\ref{thm:no_BV} and analyse the blow-up limits obtained by taking subRiemannian dilations of $E$ at a certain point $x$. We obtain distinct limits (in the sense of Caccioppoli sets) when we consider distinct sequences of dilating factors: in other words, $E$ does not admit a unique subRiemannian tangent at $x$. We also observe that the subRiemannian lower and upper densities of $E$ at $x$ are distinct ($\Theta_*(E,x) \neq \Theta^*(E,x)$, with notation as in Proposition~\ref{prop:densities}).

\begin{theorem}
 \label{prop:non_unique_tangents}
There exists a set $E$ with constant horizontal normal that admits distinct subRiemannian blow ups at some point $x$. Moreover $\Theta_*(x,E) < \Theta^*(x,E)$.
\end{theorem}

\begin{proof}
As mentioned above, we give an example in $\mathbb{F}_{23}$. Keeping the definition of $\mathcal C$ in Theorem~\ref{thm:no_BV} unchanged, we consider $$E:= \bigcup_{t\in Z}L_{\exp(t X_5)}(\mathcal C),$$ where $Z \subset \R$ is a set containing $0$. Recall that the fact that $E$ is a set with constant horizontal normal is true independently of the choice of $Z$, thanks to Corollary~\ref{monotone_subgroup_Carnot} and by the choice of $\mathcal C$. 

We denote as usual by $\delta_r$ the subRiemannian dilation of factor $r$ about $1_G$. Recalling that dilations are homomorphisms of $\mathbb{G}$ and that $\mathcal C$ is a cone with tip at $0$, we have $\delta_r(p \mathcal C) = \delta_r(p) \mathcal C$ and thus 
$$\delta_r(E)= \bigcup_{t\in Z_r}L_{\exp(t X_5)}(\mathcal C),$$
where $Z_r$ is the dilation of $Z \subset \R$ about $0$ by a factor $r^3$. 

We will choose $Z$ such that, at $0$, the lower density (with respect to $\mathcal{H}^1$) is strictly smaller than the upper density. In particular, we may choose $Z$ as follows. Consider an increasing sequence of natural numbers $n_1>1, n_2, n_3, ...$ defined recursively by $n_{j+1} = n_j^3$. Define 

$$Z= \{0\} \cup \left(\bigcup_{j\text{ odd}} \left(\frac{1}{n_j^3}, \frac{1}{n_j}\right]\right) \cup  \left(\bigcup_{j\text{ odd}} \left[-\frac{1}{n_j},-\frac{1}{n_j^3} \right)\right).$$
Choose $R_\ell = \sqrt[3]{\frac{1}{n_{2\ell +1}^2}}$ and $r_\ell = \sqrt[3]{\frac{1}{n_{2\ell}^2}}$ (recall that $Z_r$ is the dilation of $Z \subset \R$ about $0$ by a factor $r^3$ in $\R$). Then $Z_{R_{\ell}} \supset \left[-n_{2\ell+1}, -\frac{1}{n_{2\ell+1}}\right) \cup \{0\} \cup \left(\frac{1}{n_{2\ell+1}} , n_{2\ell+1}\right]$, which converges (increasingly) to the whole line; on the other hand, $Z_{r_{\ell}}$ contains $\{0\}$ and is disjoint from $\left[-n_{2\ell}, -\frac{1}{n_{2\ell}}\right)  \cup \left(\frac{1}{n_{2\ell}} , n_{2\ell}\right]$, so the blow-up of $Z$ along $r_\ell$ is $\{0\}$. 

We thus obtain that the blow up of $E$ at $0$ obtained using the sequence of dilations $\delta_{R_{\ell}}$ is (recall that $X_5=\partial_5$ in exponential coordinates of the second kind) 
$$E_2 = \bigcup_{t\in \R} (\mathcal C+t e_5) =\bigcup_{t\in \R}\left\{\;x\in\R^5\;:\; x_2\geq 0, \;x_4\geq 0, \;(x_5-t)^2\leq 2\;(x_2)^{3} \;x_4\;\right \}=$$ $$=\left\{\;x\in\R^5\;:\; x_2\geq 0, \;x_4\geq 0\right\}.$$
On the other hand, using the sequence of dilations $\delta_{r_{\ell}}$ the blow up obtained is (note that the intersection of $p+\mathcal C^{-1}$ with the line $(0,0,0,0,t)$ is a compact set for any choice of $p$)

$$E_1= \mathcal C=\left\{\;x\in\R^5\;:\; x_2\geq 0, \;x_4\geq 0, \; x_5 ^2\leq 2\;(x_2)^{3} \;x_4\;\right \}.$$
The two blow-ups $E_1$ and $E_2$ are distinct. Moreover,
%
$\vol(B_\rho(x,r) \cap E_1) < \vol(B_\rho(x,r) \cap E_2)$ and the two quantities are independent of $r$, we get that $\Theta_*(E,x)<\Theta^*(E,x)$. \footnote{More precisely, every blow-up of $E$ at $x$ must contain $E_1$ and be contained in $E_2$, therefore the density ratios of $E_1$ and $E_2$ at $0$ are respectively $\Theta_*(E,x)$ and $\Theta^*(E,x)$.}
\end{proof}

\begin{rem}
As a consequence of the rectifiability result in Section~\ref{rectifiabilityStep4}, one obtains that the behaviour in Theorem~\ref{prop:non_unique_tangents}, for a set with constant horizontal normal in $\mathbb{F}_{23}$, can only arise for $x$ belonging to a set of vanishing $D\uno_E$-measure, which, in view of \cite{Amb02},   has vanishing $\mathcal{H}^9$-measure.
\end{rem}

\begin{remark}
It is not true that every blow-up for a constant-normal set is a cone. Indeed, we may modify the above example by choosing $Z$ as follows. Pick an increasing sequence of natural numbers $n_1, n_2, n_3, ...$ with the property that $n_{j+1} > n_j^3$. Define

$$Z= \{0\} \cup \left(\bigcup_{j=1}^\infty \left(\frac{1}{n_j^2}, \frac{1}{n_j}\right]\right) \cup \{0\} \cup \left(\bigcup_{j=1}^\infty \left[-\frac{1}{n_j},-\frac{1}{n_j^2} \right)\right).$$
Choose $R_\ell=\sqrt[3]{\frac{1}{n_\ell}}$. When we dilate by $R_\ell$ we have that $Z_{R_\ell}$ contains $ \left[-1, -\frac{1}{n_\ell}\right) \cup \{0\} \cup \left(\frac{1}{n_\ell}, 1\right] $ and does not intersect $ \left(-n_\ell, -1\right)  \cup \left(1, n_\ell \right)$. This implies that, with this choice of $Z$, the blow-up of $E:= \bigcup_{t\in Z}L_{\exp(t X_5)}(\mathcal C)$ for the sequence of dilations $\delta_{R_{\ell}}$ is $\bigcup_{t\in [-1,1]} (\mathcal C+t e_5)$, which is not a cone.
\end{remark}

\section{Intrinsic rectifiability in step at most 4}
\label{rectifiabilityStep4}
\begin{definition}[Vertical half-spaces]
	We say that a measurable set $ H  $ of a Carnot group $\G $ is a vertical half-space if 
	it has constant normal and it is $Z$-monotone for all $Z\in[{\rm Lie}(G),{\rm Lie}(G)]$. 
\end{definition}

\begin{theorem}[De Giorgi, FSSC]
Let $E\subset \G$ be a finite-perimeter set in a Carnot group. 
If for $D\uno_E$-almost every $x\in \G$ we have that  every tangent $F$  of $E$ at $x$ is a vertical half-space, then $E$ is intrinsically rectifiable, in the sense of \cite{fssc}.
\end{theorem}

\begin{proposition}
	Let $\G$ be a Carnot group and let $E\subseteq G$ be a set with constant   normal. 
	If the  step of $\G$ is at most $4$, then  
	 for $D\uno_E$-almost every $x\in \G$ we have that  every tangent $F$  of $E$ at $x$ is a vertical half-space. Consequently, the set $E$ is intrinsically rectifiable.
\end{proposition}

\begin{proof}
In this proof we use the following terminology: for all vectors $X$ in the Lie algebra of $\G$ we say that $X$ is a monotone direction for $E$ if $X\uno_{E}\geq 0$; we say that 
$X$ is an invariant direction for $E$ if $X\uno_{E}= 0$.
We fix a basis  $X_1, X_2,\ldots, X_m$  of the first layer $V_1$ of $\G$ such that 
$X_1 $ is a monotone direction and $  X_2,\ldots, X_m$ are invariant directions.

Since for all $t\in \R$ and all $j=2,\ldots, m$ the vector $tX_j $ is an invariant direction (for $E$) and $X_1$ is a monotone direction, then by \cite[Proposition 4.7.(ii)]{AKL}, we have that   the vector
\begin{equation}
\label{estate2019}
{\rm Ad}_{\exp(tX_j)}X_1 = e^{{\rm ad}_{tX_j}}X_1=
 X_1 + t[X_j,X_1] +\frac{t^2}{2} [X_j,[X_j,X_1] ] +\frac{t^3}{3} [X_j,[X_j,[X_j,X_1] ] ]\end{equation}
  is a monotone direction, 
where we have used that  the step of  $\G$ is  at most 4.
Dividing \eqref{estate2019} by $t^3$ and letting $t\to\pm\infty$, we deduce that 
$[X_j,[X_j,[X_j,X_1] ] ]$ is an invariant direction. 
Therefore, going back to \eqref{estate2019} we have that
$$ X_1 + t[X_j,X_1] +\frac{t^2}{2} [X_j,[X_j,X_1] ]  ,$$
 is a monotone direction for all $t$.
Dividing by $t^2$ and letting $t\to+\infty$, we deduce that 
 $[X_j,[X_j,X_1] ] $  is a monotone direction for $E$.
 
 By  \cite[Lemma 5.8]{AKL}, we have that for   $D\uno_E$-almost every $x\in \G$ and for every tangent $F$ of $E$ at $x$ the vector 
  $[X_j,[X_j,X_1] ] $  is an invariant direction for $F$, for all $j=2,\ldots, m$. Recall that moreover, every such $F$ has the same constant normal as $E$.
  Therefore, similarly as in \eqref{estate2019} we then get that
  $$ X_1 + t[X_j,X_1] , \qquad \forall t\in \R, $$
 is a monotone direction for $F$. 
   Letting $t\to\pm\infty$, we get that 
    $[X_j,X_1] $ is an invariant direction. 
    In addition, recall that every element $[X_j,X_k] $, for $j,k=2,\ldots, m$,    is an invariant direction (since invariant directions for a subalgebra, see  \cite[Proposition 4.7.(i)]{AKL}).
  We therefore found a  set of of invariant directions that span the second layer $V_2$. Hence, we have 
  $$Z \uno_{F}= 0, \qquad \forall Z\in V_2.$$
  
 Fix $Z\in V_2$ and $i=1,\ldots, m$. Then we have 
  $Z \uno_{F}= 0$,  $X_i \uno_{F}\geq 0$, and 
  $[Z,[Z,X_i]]=0$, since we are in step at most 4.
Consequently,   from  \cite[Proposition 4.7]{AKL} we have that for all $t\in \R$ the vector
$$ X_i + t[Z,X_i]   $$
 is a monotone direction for $F$.
    Letting $t\to\pm\infty$, we get that 
    $[Z,X_i] $ is an invariant direction.  Since such vectors span $V_3$ we get that 
    $$W \uno_{F}= 0, \qquad \forall W\in V_3.$$
    
    Similarly, for $W\in V_3$ and $i=1,\ldots, m$ we get that 
    $ X_i + t[W,X_i]   $
 is a monotone direction for $F$ for all $t$, and so is $[W,X_i]$. Thus every vector in $V_4$ is an invariant direction for $F$. In conclusion, the set $F$ is a half-space.
\end{proof}

   \bibliography{general_bibliography}
\bibliographystyle{amsalpha}

\end{document}